\documentclass[11pt]{article}
\usepackage{amssymb,amsmath,amsfonts,amsthm,latexsym,amscd,mathrsfs,mathtools}
\usepackage{enumitem,microtype}
\usepackage{euscript, epic, eepic,epsfig,color}
\usepackage{graphicx}
\usepackage{verbatim}
\usepackage{fourier}
\usepackage{tikz}

\usepackage[colorlinks,linkcolor=black,anchorcolor=black,citecolor=black,hyperindex=true,CJKbookmarks=true]{hyperref}

\makeatletter
\newcommand{\subsectionruninhead}{\@startsection{subsection}{2}{0mm}
{-\baselineskip}{-0mm}{\bf\large}}
\newcommand{\subsubsectionruninhead}{\@startsection{subsubsection}{3}{0mm}
{-\baselineskip}{-0mm}{\bf\normalsize}}
\makeatother

\newtheorem*{theorem*}{Theorem}

\newtheorem{theoremalph}{Theorem}
\newtheorem{corollaryalph}[theoremalph]{Corollary}

\newtheorem*{proposition*}{Proposition}
\newtheorem*{corollary*}{Corollary}
\newtheorem*{claim*}{Claim}
\newtheorem*{remark*}{Remark}
\newtheorem*{problem*}{Problem}
\newtheorem*{question*}{Question}
\newtheorem{theorem}{Theorem}[section]

\newtheorem{proposition}[theorem]{Proposition}

\newtheorem{lemma}[theorem]{Lemma}

\theoremstyle{definition}
\newtheorem{definition}[theorem]{Definition}
\newtheorem{remark}[theorem]{Remark}

\numberwithin{equation}{section}

\newcommand{\eps}{\varepsilon}
\newcommand{\R}{\mathbb{R}}
\newcommand{\bbR}{\mathbb R}
\newcommand{\bbZ}{\mathbb Z}

\newcommand{\htop}{h_{\text{top}}}

\allowdisplaybreaks

\begin{document}

\title{Complete realization of multifractal entropy spectra and pressure functions}

\author{{Xiaobo Hou$^{1,\S}$, Wanshan Lin$^{2,\dag}$ and Xueting Tian$^{2,\ddag}$}\\
	{\em\small$^1$ School of Mathematical Science,   Dalian University of Technology}\\
	{\em\small Dalian 116024, People's Republic of China}\\
	{\em\small$^2$ School of Mathematical Science,  Fudan University}\\
	{\em\small Shanghai 200433, People's Republic of China}\\
	{\small $^\S$Email: xiaobohou@dlut.edu.cn;
	$^{\dag}$Email: wanshanlin@fudan.edu.cn}\\
	{\small $^{\ddag}$Email: xuetingtian@fudan.edu.cn}\\}



\maketitle

\begin{abstract}
We give a complete characterization of the multifractal
entropy spectra arising from continuous vector-valued potentials on transitive two-sided shifts of finite type.  We prove that, in
every finite dimension, any nonnegative upper semicontinuous concave function
on a compact convex set that attains the topological entropy as its
maximum at a unique point is realized as the entropy spectrum of a potential
whose rotation set is precisely that set.  Every such spectrum moreover
admits arbitrarily many pairwise non-cohomologous realizations.

Via Legendre--Fenchel duality, this characterization yields complete pressure flexibility over the entire parameter space.  In
particular, it resolves the whole-space problem posed by Kucherenko and Quas \cite{KQ2022}.  A separate
construction based on entropy paths extends scalar spectrum and pressure
realization to a substantially broader class of dynamical systems.  Finally,
with respect to the closed-graph Hausdorff metric, we prove that the spectrum
map is lower semicontinuous in every finite dimension, whereas upper
semicontinuity fails on a dense set for scalar potentials on transitive
shifts of finite type.
\end{abstract}

\tableofcontents

\section{Introduction}
\qquad A natural motivation for our problem comes from the \emph{flexibility program} in dynamics, initiated by Katok, which asks whether dynamical invariants can realize all values compatible with their general constraints within a prescribed class of systems. Considerable progress has been made in this direction for several important invariants, including entropy, Birkhoff averages, Lyapunov exponents, and pressure functions; see, for instance,  \cite{Katok1980,HuangXuXu2021,LiShiWangWang2020,Sun2025,Burguet2020,ChandgotiaMeyerovitch2021,QuasSoo2016,KQ2022,GuanSunWu2017,BochiKatokRodriguezHertz2021,Ures2012,LiOprocha2018,YangZhang2020,KoniecznyKupsaKwietniak2018,DGR2019,TianWangWang2019,DHT,HouTian2023} and the references therein. Most of these results concern numerical invariants. We instead study a functional invariant: the multifractal entropy spectrum associated with a continuous vector-valued potential.

This question arises naturally from multifractal analysis. Multifractal analysis investigates level sets defined by asymptotic quantities, such as Birkhoff averages, Lyapunov exponents, and local dimensions, and quantifies their size by entropy, dimension, or related characteristics; see, for instance, \cite{Fan2021,HeWolf2022,BJKR2021,Climenhaga2013,JohanssonJordanObergPollicott2010,Jenkinson2001,DGR2019,IommiJordan2015,JT2021,JordanRams2021,PfisterSullivan2007,LiuLiu2025,BS2001,BarreiraHolanda2021,KucherenkoWolf2014,PesinSadovskaya2001}. The associated spectra record how these asymptotic behaviors are distributed throughout the system.

Let $(X,f)$ be a dynamical system, where $X$ is a compact metric space and
$f:X\to X$ is continuous.  For $m\geq1$, let $C(X,\mathbb R^m)$ denote the
space of continuous maps from $X$ to $\mathbb R^m$, endowed with the uniform
norm
$
\|\Phi\|_\infty:=\sup_{x\in X}|\Phi(x)|.
$
An element $\Phi\in C(X,\mathbb R^m)$ is called a continuous vector-valued
potential.  Given $\alpha\in\mathbb R^m$, define the
\emph{Birkhoff level set}
$$
L(\Phi,\alpha;X,f)
:=
\left\{x\in X:
\lim_{n\to\infty}\frac1n\sum_{k=0}^{n-1}\Phi(f^kx)=\alpha
\right\},
$$
and the \emph{rotation set}
$$
R(\Phi;X,f)
:=
\left\{\alpha\in\mathbb R^m:L(\Phi,\alpha;X,f)\ne\emptyset\right\}.
$$
The associated multifractal entropy spectrum is
$$
E_\Phi(\alpha)
:=
\htop(f,L(\Phi,\alpha;X,f)),
\qquad\text{for every }\alpha\in R(\Phi;X,f),
$$
where $\htop(f,Y)$ denotes Bowen topological entropy of a set $Y\subset X$.
To clarify the terminology used throughout, write
$\Phi=(\varphi_1,\ldots,\varphi_m)$.  We call the general $m$-dimensional
setting \emph{multiparameter}.  The term
\emph{scalar} refers specifically to $m=1$: the potential is then a
real-valued function, usually denoted by $\varphi\in C(X)=C(X,\mathbb R)$,
and its rotation set is an interval in $\mathbb R$.  All dimensions in this
paper are finite.

Classical multifractal analysis begins with a prescribed potential and
studies the entropy, pressure, or dimension of its Birkhoff level sets.  This
direct problem leads to variational formulas, regularity properties of the
resulting spectra, and Legendre--Fenchel duality with pressure.  We consider
the complementary inverse problem: rather than fixing a potential and
determining its spectrum, we prescribe the entire entropy spectrum and ask
whether it can be realized by a continuous potential.  Our aim is therefore
not another regularity or variational theorem for spectra of fixed
potentials, but a characterization of all admissible spectral shapes.

Such spectra obey strong constraints.  Under the standard
hypotheses used below, the rotation set is compact and convex, while the
entropy spectrum is nonnegative, upper semicontinuous, and concave; its
maximum is bounded above by the topological entropy of the system.  This
leads to a flexibility problem for the entire spectrum, rather than for an
individual numerical invariant.  To the best of our knowledge, despite the
extensive development of multifractal analysis, the following realization
problem has not previously been addressed:
\begin{quote}
\itshape
Subject to these general constraints, which functions arise as multifractal
entropy spectra of continuous vector-valued potentials within a prescribed class of
dynamical systems?
\end{quote}

We use the following natural target class.
\begin{definition}\label{def:target-class}
	Let $m\geq1$. Denote by $C_{ms}^{m}$ the set of all functions
	$h:K_h\to[0,\infty)$ such that:
	\begin{enumerate}
		\item $K_h\subset\mathbb R^m$ is a nonempty compact convex set;
		\item $h$ is upper semicontinuous and concave;
		\item $h$ attains a strictly positive maximum at a unique point of $K_h$.
	\end{enumerate}
	For $H>0$, define
	$$
	C_{ms}^{m,H}
	:=
	\{h\in C_{ms}^{m}:\max_{\alpha\in K_h}h(\alpha)=H\}.
	$$
\end{definition}

Throughout, we restrict attention to the case where $K_h$ is not a
singleton, since otherwise $h$ is simply a constant function on its
singleton domain and can be realized by a constant potential.

Proposition~\ref{prop:why-ms} shows that these restrictions are necessary.
Under the saturation and upper-semicontinuity hypotheses used below, every
rotation set is compact and convex, and every entropy spectrum is upper
semicontinuous and concave.  Moreover, when the measure of maximal entropy is
unique, the spectrum attains the topological entropy at the unique rotation
vector of that measure.

When $m=1$, every compact convex set is an interval, and every finite upper
semicontinuous concave function on such an interval is automatically
continuous.  This automatic continuity is specific to the scalar setting:
in dimension two, Wolf constructed a Lipschitz potential on a full shift
whose entropy spectrum is discontinuous at a boundary point of its
rotation set~\cite{Wolf2020}.  Thus upper semicontinuity is the natural
regularity condition in higher dimensions.

Our main results address the realization problem in four directions.  First,
on every nontrivial transitive two-sided SFT, the necessary constraints above
are sufficient in every finite dimension, and each admissible spectrum has
arbitrarily many realizations modulo cohomology.  Second, Legendre--Fenchel
duality converts this result into complete whole-space pressure flexibility
and resolves a question of Kucherenko and Quas.  Third, an entropy-path
construction extends scalar spectrum and pressure realization to a much
broader class of systems.  Finally, in the closed-graph Hausdorff metric, the
spectrum map is lower semicontinuous in every finite dimension, whereas in
the scalar SFT setting upper semicontinuity fails on a dense set.  We now
state these results precisely.

\subsection{Complete multiparameter spectrum realization on transitive SFTs}

Our main theorem shows that the preceding necessary conditions are also
sufficient on every nontrivial transitive two-sided SFT, in every finite
dimension.  Recall that two vector potentials $\Phi$ and $\Psi$ are
cohomologous if
$
\Psi-\Phi=u-u\circ\sigma+c
$
for some $u\in C(X,\mathbb R^m)$ and $c\in\mathbb R^m$.

\begin{theoremalph}\label{mp:thm:main-realization}
Let $m\geq1$, let $(X,\sigma)$ be a nontrivial transitive two-sided shift
of finite type, and put $H=\htop(\sigma,X)$. Then for every
$h\in C_{ms}^{m,H}$ and every $N\in\mathbb N$, there exist
$
\Phi_1,\dots,\Phi_N\in C(X,\mathbb R^m)
$
such that, for every $j=1,\dots,N$,
$$
R(\Phi_j;X,\sigma)=K_h,
\qquad
\htop(\sigma,L(\Phi_j,\alpha;X,\sigma))=h(\alpha)
\quad\text{for every }\alpha\in K_h,
$$
and such that $\Phi_i$ and $\Phi_j$ are not cohomologous whenever $i\ne j$.
\end{theoremalph}

Theorem~\ref{mp:thm:main-realization} shows that the formal restrictions in
Definition~\ref{def:target-class} are also sufficient.  Once the ambient
entropy $H=\htop(\sigma,X)$ is fixed, every nonnegative upper
semicontinuous concave function on a nonempty compact convex set, with
maximum $H$ attained at a unique point, occurs as the multifractal entropy
spectrum of a continuous vector-valued potential.  Thus the theorem gives
an exact characterization: the dimension and the compact convex rotation
set are arbitrary, boundary discontinuities are allowed, and every
admissible spectrum has arbitrarily many realizations modulo cohomology.

\subsection{Whole-space pressure realization}

One consequence of Theorem~\ref{mp:thm:main-realization} is a
whole-space pressure realization theorem.  For a continuous map $f:X\to X$
and a scalar potential $\psi\in C(X)$, we write $P_f(\psi)$ for the
topological pressure of $\psi$ with respect to $f$.  Kucherenko and
Quas~\cite{KQ2022} established the following pressure flexibility theorem on
a positive orthant.

\begin{theorem}\label{thm-QS}
	Let $m\geq1$ and $a>0$, and let
	$F:(a,\infty)^m\to\mathbb R$ be convex and Lipschitz. Assume that the
	vertical intercepts of all supporting hyperplanes to the graph of $F$
	lie in a closed interval $[b,c]\subset[0,\infty)$. Then there exist a
	full shift $(X,\sigma)$ on a finite alphabet and continuous potentials
	$\varphi_1,\dots,\varphi_m\in C(X)$ such that
	$$
	P_\sigma(t_1\varphi_1+\cdots+t_m\varphi_m)
	=
	F(t_1,\dots,t_m)
	$$
	for every $(t_1,\dots,t_m)\in(a,\infty)^m$.
\end{theorem}

They then asked whether the parameter domain can be enlarged to the whole
space while prescribing the full shift in advance.

\begin{question*}\cite[Question~1]{KQ2022}
	Let $m\geq1$, and let $F:\mathbb R^m\to\mathbb R$ be convex and
	Lipschitz. Assume that the vertical intercepts of its supporting
	hyperplanes lie in a bounded subinterval of $\mathbb R_{\geq0}$ and that
	$F(0)=\log d$ for some integer $d\geq2$. Do there exist continuous
	potentials $\varphi_1,\dots,\varphi_m\in C(\Sigma_d)$ on the full shift
	$\Sigma_d=\{1,\dots,d\}^{\mathbb Z}$ such that
	$$
	P_\sigma(t_1\varphi_1+\cdots+t_m\varphi_m)
	=
	F(t_1,\dots,t_m)
	$$
	for every $(t_1,\dots,t_m)\in\mathbb R^m$?
\end{question*}

For any pressure function realized on a transitive SFT, uniqueness of the
measure of maximal entropy forces differentiability at the origin; see
Section~\ref{sec:differentiability-at-origin}.  This necessary condition is
absent from the question above.  By applying
Theorem~\ref{mp:thm:main-realization} to the concave dual of $F$, we show that
it is the only missing condition.

\begin{theoremalph}\label{mp:thm:whole-space-pressure}
	Let $m\geq1$, let $(X,\sigma)$ be a nontrivial transitive two-sided shift
	of finite type, and put $H=\htop(\sigma,X)$. Let
	$F:\mathbb R^m\to\mathbb R$ be convex, and assume
	that:
	\begin{enumerate}
		\item $F(0)=H$;
		\item the vertical intercepts of all supporting hyperplanes of $F$ are
		nonnegative;
		\item $F$ is differentiable at $0$.
	\end{enumerate}
	Then for every $N\in\mathbb N$ there exist
	$\Phi_1,\dots,\Phi_N\in C(X,\mathbb R^m)$ such that for every $j=1,\dots,N$,
	$$
	P_\sigma(t\cdot\Phi_j)=F(t)
	\qquad\text{for every }t\in\mathbb R^m
	$$
    and such that $\Phi_i$ and $\Phi_j$ are not cohomologous whenever $i\ne j$.
\end{theoremalph}

For the full $d$-shift, where $H=\log d$,
Theorem~\ref{mp:thm:whole-space-pressure} answers the question of Kucherenko
and Quas after adding the necessary differentiability condition.  The result
is stronger in two respects: it applies to every nontrivial transitive
two-sided SFT and provides arbitrarily many realizations modulo cohomology.

\begin{remark}
	It is natural to ask whether Theorem~\ref{mp:thm:main-realization} could
	be deduced from the pressure flexibility theorem of Kucherenko and Quas,
	Theorem~\ref{thm-QS}, by Legendre--Fenchel duality.  This is not the
	case.  Theorem~\ref{thm-QS} prescribes the pressure only on a positive
	orthant $(a,\infty)^m$ with $a>0$, leaving its behavior near the origin
	and in all other parameter directions undetermined.  The missing
	directions are essential for recovering the entropy spectrum on the
	whole rotation set.  Thus the theorem of Kucherenko and
	Quas does not imply the spectrum realization theorem.  Our argument
	proceeds in the reverse direction: we first realize the entropy spectrum
	and then apply the variational principle and convex duality to obtain the
	whole-space pressure realization in
	Theorem~\ref{mp:thm:whole-space-pressure}.
\end{remark}

\subsection{Scalar spectrum and pressure realization beyond SFTs}

The multiparameter construction uses symbolic structures specific to SFTs.
In the scalar case, however, a different construction based on
entropy monotone paths gives realization on a much broader class of systems.
The hypotheses ${\rm(H1)}$--${\rm(H5)}$ and their precise roles are stated
in Section~\ref{sec:scalar-realization}. Examples include transitive
sofic shifts, $\beta$-shifts, $S$-gap shifts, transitive Anosov
diffeomorphisms, and transitive expanding maps.

\begin{theoremalph}\label{thm:main}\label{thm:many-realizations}
	Let $(X,f)$ be a dynamical system satisfying ${\rm(H1)}$--${\rm(H4)}$, and put $H=\htop(f,X)$.
	Then for any $h\in C_{ms}^{1,H}$, there exists
	$\varphi\in C(X)$ such that
	$$
	R(\varphi;X,f)=K_h,
	\qquad
	\htop(f,L(\varphi,\alpha;X,f))=h(\alpha)
	\quad\text{for every }\alpha\in K_h.
	$$
	If, in addition, $(X,f)$ satisfies ${\rm(H5)}$, then for every
	$N\in\mathbb N$ there are $N$ pairwise non-cohomologous continuous
	potentials with the same rotation interval and entropy spectrum.
\end{theoremalph}

Theorems~\ref{mp:thm:main-realization} and~\ref{thm:main} are complementary.
The former allows arbitrary finite dimension and arbitrary compact convex
rotation sets but uses symbolic structures specific to SFTs.  The latter is
scalar but applies to a substantially larger class of systems.  Both follow
from the same calibrated-map realization principle, using different
calibrated pairs.

The scalar theorem also yields pressure flexibility on the whole line.
\begin{corollaryalph}\label{cor:KQ-strengthened}
	Let $H>0$, let $F:\mathbb R\to\mathbb R$ be convex, satisfy
	$F(0)=H$, be differentiable at $0$, and have nonnegative vertical
	intercepts for all supporting lines.  If $(X,f)$ has
	$\htop(f,X)=H$ and satisfies ${\rm(H1)}$--${\rm(H4)}$, then there is
	$\varphi\in C(X)$ such that $P_f(t\varphi)=F(t)$ for every
	$t\in\mathbb R$.  Under ${\rm(H5)}$, there are arbitrarily many pairwise
	non-cohomologous such realizations.
\end{corollaryalph}

\subsection{Stability properties of the spectrum map}
\label{subsec:intro-spectrum-stability}

The realization theorems describe the range of the spectrum map.  We also
study its local behavior under perturbations.  Since the domain of a spectrum
depends on the potential, we identify each spectrum with its closed graph.
For a function $h:K_h\to\mathbb R$, where $K_h\subset\mathbb R^m$ is
compact, set
$$
\Gamma_h:=\{(\alpha,h(\alpha)):\alpha\in K_h\},
\qquad
\widehat\Gamma_h:=\overline{\Gamma_h}\subset\mathbb R^{m+1}.
$$
The closure is relevant only in higher dimensions: scalar entropy spectra
are continuous on their rotation intervals, whereas boundary discontinuities
may occur when $m\geq2$.

For $\Phi\in C(X,\mathbb R^m)$, write $E_m(\Phi):=E_\Phi$ and define the
\emph{multifractal entropy spectrum map} by
$$
\mathscr E_m:C(X,\mathbb R^m)\longrightarrow
\mathcal K(\mathbb R^{m+1}),
\qquad
\mathscr E_m(\Phi):=\widehat\Gamma_{E_m(\Phi)},
$$
where $\mathcal K(\mathbb R^{m+1})$ denotes the collection of nonempty
compact subsets of $\mathbb R^{m+1}$.

For nonempty compact sets $A,B\subset\mathbb R^{m+1}$, let
$$
\operatorname{dist}(a,B):=\inf_{b\in B}|a-b|,
\qquad
e(A,B):=\sup_{a\in A}\operatorname{dist}(a,B).
$$
The Hausdorff distance is
$
d_H(A,B)=\max\{e(A,B),e(B,A)\},
$
and the corresponding closed-graph distance between spectra is
$
d_{ms}^{(m)}(h_1,h_2)
:=d_H(\widehat\Gamma_{h_1},\widehat\Gamma_{h_2}).
$
We call $\mathscr E_m$ lower semicontinuous at $\Phi$ if
$
e\bigl(\mathscr E_m(\Phi),\mathscr E_m(\Phi_n)\bigr)\longrightarrow0
$
whenever $\Phi_n\to\Phi$ uniformly.  Upper semicontinuity is defined by
reversing the two arguments of $e$, and continuity with respect to
$d_{ms}^{(m)}$ is equivalent to both one-sided conditions.

\begin{theoremalph}\label{thm:lower-continuous}
	Let $(X,f)$ be a dynamical system satisfying ${\rm(H1)}$ and ${\rm(H3)}$.
	Then, for every $m\geq1$, the map $\mathscr E_m$ is lower semicontinuous
	at every $\Phi\in C(X,\mathbb R^m)$.
\end{theoremalph}

\begin{theoremalph}\label{thm:not-continuous}
	Let $(X,\sigma)$ be a nontrivial transitive two-sided shift of finite
	type.  There is a dense set $D\subset C(X)$ such that $\mathscr E_1$ is
	not upper semicontinuous, and hence not continuous, at any
	$\varphi\in D$.
\end{theoremalph}

\subsection*{Ideas of the proofs}

The conceptual core of the paper is a realization principle based on
calibrated maps.  The principle assumes a compact set
$C\subset\mathcal M_f^e(X)$ and a continuous map $G:C\to K_h$ satisfying
$$
h_\mu(f)\leq h(G(\mu))
$$
for every $\mu\in C$, together with an exact calibrated measure over each
$\alpha\in K_h$.  The face generated by $C$ is a Bauer simplex, on which
$G$ extends continuously and affinely.  Lazar's affine selection theorem
then extends this map to the full invariant measure simplex.  Representing
the resulting affine map by a vector potential and applying the conditional
variational principle produces the prescribed spectrum.

The two realization theorems differ in how the calibrated pair is built.  In
the multiparameter SFT setting, prescribed-entropy embeddings, pairwise
separated high entropy SFT layers, and a shell decomposition of the hypograph
of $h$ produce $C$ and $G$.  In the scalar setting, the same role is played by
an entropy monotone path joining two zero entropy measures through the measure
of maximal entropy.  Once the calibrated pair is available, the remainder of
the argument is common to both settings.

\subsection*{Structure of the paper}

Section~\ref{Section:Preliminaries} reviews the basic properties of
multifractal entropy spectra and the convex analytic tools.  Section~\ref{sec:calibrated-realization} establishes the
calibrated map realization principle underlying both constructions.
Section~\ref{sec:multiparameter-realization} implements this principle in the
multiparameter symbolic setting and proves
Theorem~\ref{mp:thm:main-realization}.  Section~\ref{sec:app} uses
Legendre--Fenchel duality to obtain whole-space pressure flexibility and
settle the Kucherenko--Quas problem.  Section~\ref{sec:scalar-realization}
develops the scalar entropy-path construction and extends the realization
results beyond SFTs.  Section~\ref{Section:discontinuity} analyzes the
stability of the spectrum map, proving lower semicontinuity in every finite
dimension and dense failure of upper semicontinuity in the scalar SFT
setting.  The appendix establishes the local SFT approximation required for
the multiparameter construction.

\section{Multifractal entropy spectra and convex analytic tools}\label{Section:Preliminaries}
\subsection{Basic properties of multifractal entropy spectra}
For a compact metric space $X$, denote by $\mathcal M(X)$ the set of Borel probability measures on $X$. If $f:X\to X$ is continuous, denote by $\mathcal M_f(X)$ and $\mathcal M_f^e(X)$ the sets of $f$-invariant and $f$-ergodic Borel probability measures, respectively. We always endow $\mathcal M(X)$ with the weak$^*$ topology. It follows from \cite{Walters1982} that $\mathcal M_f(X)$ is a nonempty compact convex set, and its extreme points are precisely $\mathcal M_f^e(X)$.

\begin{definition}
	For any $\mu\in \mathcal M_f(X)$, let $h_\mu(f)$ denote \emph{the metric entropy} of $\mu$ with respect to $f$; see \cite{Walters1982} for the definition.
	
	For any subset $Y\subset X$, let $\htop(f,Y)$ denote \emph{the topological entropy} of $Y$ in the sense of Bowen; see \cite{Bowen1973} for the definition. For convenience, we write $\htop(f)=\htop(f,X)$. 
	
	For any $\varphi\in C(X)$, let $P_f(\varphi)$ denote \emph{the topological pressure} of $\varphi$ with respect to $f$; see \cite{Bowen1973} for the definition. In particular,
	$
	P_f(0)=\htop(f).
	$
	
	For any $\mu\in\mathcal M_f(X)$, let $\xi_\mu$ denote its ergodic decomposition measure, namely the unique Borel probability measure on $\mathcal M_f^e(X)$ satisfying, for every $\varphi\in C(X)$,
	$$
	\int_X\varphi\,d\mu
	=
	\int_{\mathcal M_f^e(X)}
	\left(
	\int_X\varphi\,d\nu
	\right)
	\,d\xi_\mu(\nu).
	$$
	See \cite{Walters1982} for details. The ergodic decomposition formula for entropy is
	$$
	h_\mu(f)
	=
	\int_{\mathcal M_f^e(X)}h_\nu(f)\,d\xi_\mu(\nu).
	$$
	
	We say an invariant measure is uniquely ergodic if $\mathcal{M}_f(\operatorname{Supp}(\mu))=\{\mu\}$, where $\operatorname{Supp}(\mu)$ is the support of $\mu$.
\end{definition}

\begin{definition}
	For $\varphi\in C(X)$, define
	$$
	\mathcal M_{\max}(\varphi)
	:=
	\left\{
	\mu\in\mathcal M_f(X):
	\int\varphi\,d\mu
	=
	\max_{\nu\in\mathcal M_f(X)}\int\varphi\,d\nu
	\right\}.
	$$
	Its elements are called the \emph{maximizing measures} of $\varphi$.
\end{definition}

We use the following properties. See, for example, \cite{Pesin1997,Qiu2011}.
\begin{proposition}\label{prop:pressure}
	Let $(X,f)$ be a dynamical system. Then the following hold:
	\begin{enumerate}
		\item For every $\varphi\in C(X)$,
		$
		P_f(\varphi)
		=
		\sup_{\mu\in\mathcal M_f(X)}
		\left(
		h_\mu(f)+\int \varphi\,d\mu
		\right).
		$
		A measure $\nu\in\mathcal M_f(X)$ is called an equilibrium state for $\varphi$ if
		$
		P_f(\varphi)
		=
		h_\nu(f)+\int \varphi\,d\nu.
		$
		In the case $\varphi\equiv 0$, an equilibrium state is called a measure of maximal entropy. Equivalently, a measure $m\in\mathcal M_f(X)$ is a measure of maximal entropy if
		$
		h_m(f)
		=
		\sup_{\mu\in\mathcal M_f(X)}h_\mu(f)
		=
		\htop(f).
		$
		
		\item For any $\varphi_1,\varphi_2\in C(X)$,
		$
		|P_f(\varphi_1)-P_f(\varphi_2)|
		\leq
		\|\varphi_1-\varphi_2\|_\infty.
		$
		
		\item If the entropy map
		$
		\mu\longmapsto h_\mu(f)
		$
		is upper semicontinuous on $\mathcal M_f(X)$, and $\varphi$ has a unique equilibrium state $\nu_\varphi$, then for every $\psi\in C(X)$ one has
		$
		\left.\frac{d}{dt}P_{f}(\varphi+t\psi)\right|_{t=0}
		=
		\int \psi\,d\nu_\varphi.
		$
		In particular, if $(X,f)$ has a unique measure of maximal entropy $m_{\max}$, then for every $\psi\in C(X)$ one has
		$
		\left.\frac{d}{dt}P_{f}(t\psi)\right|_{t=0}
		=
		\int \psi\,dm_{\max}.
		$

		\item Suppose that the entropy map is upper semicontinuous on
		$\mathcal M_f(X)$, and let
		$
		\mathcal R_f
		:=
		\{\varphi\in C(X):\varphi\text{ has a unique equilibrium state}\}.
		$
		For $\varphi\in\mathcal R_f$, denote its unique equilibrium state by
		$\nu_\varphi$. Then the map
		$$
		\mathcal R_f\longrightarrow\mathcal M_f(X),
		\qquad
		\varphi\longmapsto\nu_\varphi,
		$$
		is continuous from the uniform topology to the weak$^*$ topology.
	\end{enumerate}
\end{proposition}

\begin{definition}\label{def:saturated-property}
	For any $x \in X$ and $n\in\mathbb{N}$, define
	$
	\mathcal{E}_n(x):=\frac{1}{n} \sum_{k=0}^{n-1} \delta_{f^k (x)}
	$,
	where $\delta_y$ is the Dirac mass at $y \in X$. For
	$\mu\in\mathcal M_f(X)$, define its \emph{generic set} by
	$
	G_\mu:=\{x\in X:\lim\limits_{n\to\infty} \mathcal{E}_n(x)=\mu\}.
	$
	We say that a dynamical system $(X,f)$ has \emph{the saturated property
	for every invariant measure} if $G_\mu\neq\emptyset$ and
	$\htop(f,G_\mu)=h_\mu(f)$ for every $\mu\in\mathcal{M}_f(X)$. 
\end{definition}

\begin{proposition}\label{prop:shifts-facts}
	Let $m\geq1$. If $(X,f)$ has the saturated property for every
	invariant measure, then for every $\Phi\in C(X,\mathbb R^m)$,
	$$
	R(\Phi;X,f)
	=
	\left\{\int\Phi\,d\mu:\mu\in\mathcal M_f(X)\right\}
	$$
	is a nonempty compact convex subset of $\mathbb R^m$, and for every
	$\alpha\in R(\Phi;X,f)$,
	$$
	\htop(f,L(\Phi,\alpha;X,f))
	=
	\sup\left\{h_\mu(f):
	\mu\in\mathcal M_f(X),\ \int\Phi\,d\mu=\alpha
	\right\}.
	$$
	The latter identity is referred to as the \emph{conditional variational
	principle} for the level sets of $\Phi$.
\end{proposition}
\begin{proof}
	Fix $\Phi\in C(X,\mathbb R^m)$.  If $\alpha\in R(\Phi;X,f)$, choose
	$x\in L(\Phi,\alpha;X,f)$.  Every weak$^*$ accumulation point $\mu$ of
	$(\mathcal E_n(x))_{n\geq1}$ belongs to $\mathcal M_f(X)$ and satisfies
	$
	\int\Phi\,d\mu=\alpha.
	$
	This proves $
	R(\Phi;X,f)
	\subset 
	\left\{\int\Phi\,d\mu:\mu\in\mathcal M_f(X)\right\}
	$.
	Conversely, if $\mu\in\mathcal M_f(X)$, the saturated property gives a point
	$x\in G_\mu$.  Hence
	$$
	\frac1n\sum_{k=0}^{n-1}\Phi(f^k(x))
	\longrightarrow
	\int\Phi\,d\mu,
	$$
	so $\int\Phi\,d\mu\in R(\Phi;X,f)$.  Therefore
	$$
	R(\Phi;X,f)
	=
	\left\{\int\Phi\,d\mu:\mu\in\mathcal M_f(X)\right\}.
	$$
	The right-hand side is a nonempty compact convex set because it is the image
	of the nonempty compact convex set $\mathcal M_f(X)$ under the continuous
	affine map $\mu\mapsto\int\Phi\,d\mu$.

	Finally, by \cite[Proposition~7.1]{PfisterSullivan2007}, for every
	$\alpha\in R(\Phi;X,f)$ one has
	$
	\htop(f,L(\Phi,\alpha;X,f))
	=
	\sup\left\{h_\mu(f):
	\mu\in\mathcal M_f(X),\ \int\Phi\,d\mu=\alpha
	\right\}.
	$
\end{proof}

The next proposition explains the choice of the target class $C_{ms}^{m}$:
every multifractal entropy spectrum is automatically upper
semicontinuous and concave on its rotation set, and uniqueness of the measure
of maximal entropy forces the maximum to be attained at exactly one point.
\begin{proposition}\label{prop:why-ms}
	Let $m\geq1$, and suppose that the entropy map is upper semicontinuous
	on $\mathcal M_f(X)$ and that $(X,f)$ has the saturated property for
	every invariant measure. Then for every $\Phi\in C(X,\mathbb R^m)$ the map
	$$
	E_\Phi(\alpha):=\htop(f,L(\Phi,\alpha;X,f)),
	\qquad\text{for every }\alpha\in R(\Phi;X,f),
	$$
	is upper semicontinuous and concave on $R(\Phi;X,f)$. If $(X,f)$ has a
	unique measure of maximal entropy $m_{\max}$, then $E_\Phi$ attains its
	maximum value $\htop(f)$ at the unique point
	$
	\alpha_\ast:=\int\Phi\,dm_{\max}.
	$
\end{proposition}
\begin{proof}
	By Proposition~\ref{prop:shifts-facts}, for every
	$\alpha\in R(\Phi;X,f)$,
	$
	E_\Phi(\alpha)
	=
	\max\left\{h_\mu(f):
	\mu\in\mathcal M_f(X),\ \int\Phi\,d\mu=\alpha\right\},
	$
	where attainment follows from the upper
	semicontinuity of the entropy map.

	To prove upper semicontinuity, let $\alpha_n\to\alpha$ in
	$R(\Phi;X,f)$. For each $n$, choose $\mu_n\in\mathcal M_f(X)$ such that
	$\int\Phi\,d\mu_n=\alpha_n$ and
	$h_{\mu_n}(f)=E_\Phi(\alpha_n)$. Passing first to a subsequence along which
	$E_\Phi(\alpha_n)$ converges to its limsup, and then to a further
	subsequence, compactness gives $\mu_n\to\mu\in\mathcal M_f(X)$.
	Since $\Phi$ is continuous, weak$^*$ convergence gives
	$
	\int\Phi\,d\mu
	=
	\lim_{n\to\infty}\int\Phi\,d\mu_n
	=
	\lim_{n\to\infty}\alpha_n
	=
	\alpha,
	$
	while
	upper semicontinuity of entropy gives
	$
	\limsup_{n\to\infty}E_\Phi(\alpha_n)
	=
	\limsup_{n\to\infty}h_{\mu_n}(f)
	\leq h_\mu(f)
	\leq E_\Phi(\alpha).
	$

	For concavity, let $\alpha_1,\alpha_2\in R(\Phi;X,f)$ and
	$\lambda\in(0,1)$. For each $i\in\{1,2\}$, choose
	$\mu_i\in\mathcal M_f(X)$ such that
	$\int\Phi\,d\mu_i=\alpha_i$ and
	$h_{\mu_i}(f)=E_\Phi(\alpha_i)$, and set
	$\mu=\lambda\mu_1+(1-\lambda)\mu_2$. Then
	$\int\Phi\,d\mu=\lambda\alpha_1+(1-\lambda)\alpha_2$, and affinity of
	entropy yields
	$$
	E_\Phi(\lambda\alpha_1+(1-\lambda)\alpha_2)
	\geq h_\mu(f)
	=\lambda E_\Phi(\alpha_1)+(1-\lambda)E_\Phi(\alpha_2).
	$$

	Assume now that $m_{\max}$ is the unique measure of maximal entropy, and
	put $\alpha_\ast=\int\Phi\,dm_{\max}$. The conditional variational
	principle gives
	$
	E_\Phi(\alpha_\ast)=h_{m_{\max}}(f)=\htop(f).
	$
	If $E_\Phi(\alpha)=\htop(f)$, choose $\mu\in\mathcal M_f(X)$ such that
	$\int\Phi\,d\mu=\alpha$ and $h_\mu(f)=E_\Phi(\alpha)$. Then $\mu$ is a
	measure of maximal entropy and hence equals $m_{\max}$. Therefore
	$\alpha=\alpha_\ast$, proving uniqueness.
\end{proof}

\subsection{Convex analytic tools}
Throughout this paper, all Choquet and Bauer simplices are assumed to be metrizable.

\begin{definition}\label{def:choquet-bauer-face}
Let $Q$ be a compact convex subset of a locally convex linear space.  A point $q\in Q$ is \emph{extreme} if it cannot be written as a nontrivial convex combination of two distinct points of $Q$; the set of extreme points is denoted by $\operatorname{ext}Q$.  The set $Q$ is a \emph{Choquet simplex} if every $q\in Q$ has a unique representing probability measure $\xi_q$ on $\operatorname{ext}Q$ with barycenter
$$
q=\int_{\operatorname{ext}Q}x\,d\xi_q(x).
$$
It is a \emph{Bauer simplex} if, in addition, $\operatorname{ext}Q$ is closed.  A convex subset $F\subset Q$ is a \emph{face} if
$$
0<t<1,\quad tq_1+(1-t)q_2\in F
\quad\Longrightarrow\quad
q_1,q_2\in F.
$$
\end{definition}

By the ergodic decomposition theorem recalled above,
$\mathcal M_f(X)$ is a Choquet simplex. Its extreme points
are $\mathcal M_f^e(X)$, and the representing measure of
$\mu\in\mathcal M_f(X)$ is precisely $\xi_\mu$.

We recall the following notions.
\begin{definition}
Let $K$ be a Choquet simplex and let $m\geq1$.
\begin{enumerate}
\item A set-valued map $T:K\to 2^{\mathbb R^m}$ is called \emph{affine} if
every $T(k)$ is a nonempty convex subset of $\mathbb R^m$ and
$
\lambda T(k_1)+(1-\lambda)T(k_2)
\subset
T\bigl(\lambda k_1+(1-\lambda)k_2\bigr)
$
for all $k_1,k_2\in K$ and $0<\lambda<1$.
\item The map $T:K\to 2^{\mathbb R^m}$ is called \emph{lower
semicontinuous} if, for every open set $U\subset\mathbb R^m$, the set
$
\{k\in K:T(k)\cap U\neq\varnothing\}
$
is open in $K$.
\item A map $g:K\to\mathbb R^m$ is called an \emph{affine continuous
selection} for $T$ if $g$ is affine and continuous and $g(k)\in T(k)$ for
every $k\in K$.
\end{enumerate}
\end{definition}

We shall use the following form of Lazar's affine selection theorem.
\begin{lemma}\cite[Theorem 3.1 and Corollary 3.4]{Lazar1968}\label{thm:lazar}
	Let $K$ be a Choquet simplex, let $F$ be a closed face of $K$, and let
	$
	T:K\to 2^{\R^m}
	$
	be an affine lower semicontinuous set-valued map with nonempty closed
	convex values. Suppose that $g_F:F\to\mathbb R^m$ is continuous and
	affine and satisfies $g_F(x)\in T(x)$ for every $x\in F$. Then there
	exists a continuous affine map $g:K\to\mathbb R^m$ such that
	$
	g|_F=g_F
	~\text{and}~
	g(x)\in T(x)~\text{for every }x\in K.
	$
\end{lemma}

\begin{lemma}\cite[Proposition 6]{Phelps2002}\label{lem:affine-representation}
	Let $(X,f)$ be a dynamical system. If $A:\mathcal M_f(X)\to\mathbb R$ is weak$^*$ continuous and affine, then there exists $\varphi\in C(X)$ such that
	$$
	A(\mu)=\int_X\varphi\,d\mu
	\qquad\text{for every }\mu\in\mathcal M_f(X).
	$$
	Consequently, every weak$^*$ continuous affine map $L=(L_1,\ldots,L_m):\mathcal M_f(X)\to\mathbb R^m$ is represented by a continuous vector-valued potential $\Phi=(\varphi_1,\ldots,\varphi_m)\in C(X,\mathbb R^m)$:
	$$
	L(\mu)=\int_X\Phi\,d\mu 	\qquad\text{for every }\mu\in\mathcal M_f(X).
	$$
\end{lemma}

\section{Realization from calibrated maps}\label{sec:calibrated-realization}\label{Section:beta-shifts}
This section isolates the passage from a calibrated family of ergodic measures to a realizing vector-valued potential.  It is the common abstract mechanism used by both the multiparameter SFT construction and the scalar entropy-path construction.

\subsection{Compact faces generated by ergodic measures}
For a subset $A$ of a locally convex topological vector space, we write $\operatorname{conv}(A)$ for the convex hull of $A$, namely
$$
\operatorname{conv}(A)
=
\left\{
\sum_{i=1}^n t_i x_i
:
n\geq 1,\ x_i\in A,\ t_i\geq 0,\ \sum_{i=1}^n t_i=1
\right\},
$$
and we write $\overline{\operatorname{conv}}(A)$ for its closed convex hull.

\begin{proposition}\label{prop-compact-ergodic-face}
	Let $(X,f)$ be a dynamical system. Suppose that
	$
	E\subset \mathcal{M}_f^e(X)
	$
	is weak$^*$ compact, and set
	$$
	F_E:=\overline{\operatorname{conv}(E)}\subset \mathcal{M}_f(X).
	$$
	Then the following hold:
	\begin{enumerate}
		\item $\operatorname{ext}(F_E)=E$;
		\item $F_E$ is a closed face of $\mathcal{M}_f(X)$;
		\item $F_E$ is a Bauer simplex;
		\item every continuous function $\phi:E\to\mathbb{R}$ extends uniquely to a continuous affine function $\varphi$ on $F_E$. Moreover, $\varphi$ is defined by 
		$
		\varphi(\eta)=\int_E \phi(\mu)\,d\xi_\eta(\mu)
		$
		for every 
		$\eta\in F_E$, where $\xi_\eta$ is the unique representing probability measure of $\eta$.
	\end{enumerate}
\end{proposition}

\begin{proof}
	Since 
	$
	\operatorname{ext}\bigl(\mathcal{M}_f(X)\bigr)=\mathcal{M}_f^e(X),
	$
	we have
	$
	E\subset \operatorname{ext}\bigl(\mathcal{M}_f(X)\bigr).
	$
	
	First, we show that
	$
	\operatorname{ext}(F_E)=E.
	$
	To prove $E\subset \operatorname{ext}(F_E)$, let $\mu\in E$. Suppose that
	$
	\mu=t\nu+(1-t)\omega
	$
	for some $\nu,\omega\in F_E$ and $t\in(0,1)$. Since $F_E\subset \mathcal{M}_f(X)$, this is also a convex decomposition of $\mu$ inside $\mathcal{M}_f(X)$. Because $\mu\in \mathcal{M}_f^e(X)=\operatorname{ext}(\mathcal{M}_f(X))$, we must have
	$
	\nu=\omega=\mu.
	$
	Hence $\mu\in \operatorname{ext}(F_E)$, and therefore
	$
	E\subset \operatorname{ext}(F_E).
	$
	For the reverse inclusion, since $E$ is compact and
	$
	F_E=\overline{\operatorname{conv}(E)},
	$
	Milman's converse to the Krein--Milman theorem implies that
	$
	\operatorname{ext}(F_E)\subset E;
	$
	see \cite[Proposition 1.5]{Phelps2001}. Consequently,
	$
	\operatorname{ext}(F_E)=E.
	$
	
	Next, we show that $F_E$ is a face of $\mathcal{M}_f(X)$. Let $\mu\in F_E$ and suppose that
	$
	\mu=t\nu+(1-t)\omega
	$
	for some $\nu,\omega\in \mathcal{M}_f(X)$ and $t\in(0,1)$. We claim that the ergodic decomposition of $\mu$ is supported on $E$. Indeed, choose a sequence $(\mu_n)_n$ in $\operatorname{conv}(E)$ such that $\mu_n\to \mu$ in the weak$^*$ topology. For each $n$, since $\mu_n\in\operatorname{conv}(E)$, there exist $k_n\geq 1$, measures $\nu_{n,1},\ldots,\nu_{n,k_n}\in E$, and coefficients $a_{n,1},\ldots,a_{n,k_n}\geq 0$ with
	$
	\sum_{j=1}^{k_n}a_{n,j}=1
	$
	such that
	$
	\mu_n=\sum_{j=1}^{k_n}a_{n,j}\nu_{n,j}.
	$
	Let $\mathcal P(E)$ denote the space of Borel probability measures on $E$. Define $\lambda_n\in\mathcal P(E)$ by
	$$
	\lambda_n=\sum_{j=1}^{k_n}a_{n,j}\delta_{\nu_{n,j}}.
	$$
	Since $E$ is compact, after passing to a subsequence we may assume that $\lambda_n\to \lambda\in \mathcal P(E)$ in the weak$^*$ topology. Then, for every $\varphi\in C(X)$,
	$$
	\int_X \varphi\, d\mu
	=
	\lim_{n\to\infty}\int_X \varphi\, d\mu_n
	=
	\lim_{n\to\infty}
	\sum_{j=1}^{k_n}a_{n,j}
	\int_X \varphi\, d\nu_{n,j}
	=
	\int_E
	\left(
	\int_X \varphi\, d\eta
	\right)
	\, d\lambda(\eta).
	$$
	Thus $\lambda$ is a representing measure for $\mu$ supported on $E$. By uniqueness of ergodic decomposition, the ergodic decomposition of $\mu$ is exactly $\lambda$, and hence is supported on $E$. Let $\xi_\mu,\xi_\nu$, and $\xi_\omega$
	denote the ergodic decomposition measures of $\mu,\nu$, and $\omega$, respectively. By uniqueness
	of ergodic decomposition, we have
	$
	\xi_\mu=t\xi_\nu+(1-t)\xi_\omega.
	$
	Since $\xi_\mu(E)=1$, it follows that
	$
	\xi_\nu(E)=\xi_\omega(E)=1.
	$
	Hence the ergodic decomposition measures $\xi_\nu$ and $\xi_\omega$ are supported on $E$. Therefore, for every $\varphi\in C(X)$,
	$$
	\int_X\varphi\,d\nu
	=
	\int_E
	\left(
	\int_X\varphi\,d\eta
	\right)
	\,d\xi_\nu(\eta),
	\qquad
	\int_X\varphi\,d\omega
	=
	\int_E
	\left(
	\int_X\varphi\,d\eta
	\right)
	\,d\xi_\omega(\eta).
	$$
	Since $E$ is compact, every probability measure on $E$ is a weak$^*$ limit of finitely supported probability measures. Hence there exist
	$$
	\xi_{\nu,n}
	=
	\sum_{j=1}^{k_n}a_{n,j}\delta_{\eta_{n,j}},
	\qquad
	\eta_{n,j}\in E,\quad
	a_{n,j}\geq 0,\quad
	\sum_{j=1}^{k_n}a_{n,j}=1,
	$$
	such that $\xi_{\nu,n}\to \xi_\nu$ in the weak$^*$ topology on $\mathcal P(E)$. Define
	$$
	\nu_n=\sum_{j=1}^{k_n}a_{n,j}\eta_{n,j}.
	$$
	Then $\nu_n\in\operatorname{conv}(E)$, and for every $\varphi\in C(X)$,
	$$
	\int_X\varphi\,d\nu_n
	=
	\int_E
	\left(
	\int_X\varphi\,d\eta
	\right)
	\,d\xi_{\nu,n}(\eta)
	\to
	\int_E
	\left(
	\int_X\varphi\,d\eta
	\right)
	\,d\xi_\nu(\eta)
	=
	\int_X\varphi\,d\nu.
	$$
	Thus $\nu_n\to\nu$ in the weak$^*$ topology, and hence $\nu\in\overline{\operatorname{conv}}(E)$. The same argument applied to $\xi_\omega$ gives $\omega\in\overline{\operatorname{conv}}(E)$. Hence
	$
	\nu,\omega\in \overline{\operatorname{conv}}(E)=F_E.
	$
	Thus $F_E$ is a face of $\mathcal{M}_f(X)$.
	
	Since $\mathcal M_f(X)$ is a Choquet simplex and $F_E$ is a closed face of $\mathcal M_f(X)$, it follows that $F_E$ is again a Choquet simplex; see, for example, \cite[Proposition 10.9]{Goodearl2010}. Moreover,
	$
	\operatorname{ext}(F_E)=E.
	$
	Since $\operatorname{ext}(F_E)=E$ is compact, $F_E$ is a Bauer simplex.
	
	Finally, since $F_E$ is a Bauer simplex and
	$
	\operatorname{ext}(F_E)=E,
	$
	the standard characterization of Bauer simplices implies that every continuous function on $E$ extends uniquely to a continuous affine function on $F_E$; see, for example, \cite[Theorem II. 4.3]{Alfsen1971}. More precisely, for every continuous function $\phi:E\to\mathbb R$, the unique continuous affine extension is given by
	$$
	\varphi(\eta)
	=
	\int_E \phi(\mu)\,d\xi_\eta(\mu),
	\qquad\text{for every }\eta\in F_E,
	$$
	where $\xi_\eta$ is the unique representing probability measure of $\eta$ supported on $E$.
\end{proof}

\subsection{Calibrated pairs and the abstract realization theorem}\label{subsec:common-calibrated-map}

We isolate the part of the proof shared by the scalar and multiparameter
settings.  It is formulated for a compact set of ergodic measures together
with a continuous parameter map calibrated by the target entropy function.

For a dynamical system $(X,f)$, two vector-valued potentials
$\Phi,\Psi\in C(X,\mathbb R^m)$ are called \emph{cohomologous} if there exist
$u\in C(X,\mathbb R^m)$ and $c\in\mathbb R^m$ such that
$
\Psi-\Phi=u-u\circ f+c.
$

\begin{theorem}\label{thm:calibrated-map-realization}
Let $m\geq1$, and let $(X,f)$ be a dynamical system with positive topological
entropy satisfying the following two conditions: the entropy map is upper
semicontinuous on $\mathcal M_f(X)$, and $(X,f)$ has the saturated property
for every invariant measure.  Set $H:=\htop(f)$, and let
$h\in C_{ms}^{m,H}$. Assume that there are a compact set
$C\subset\mathcal M_f^e(X)$ and a continuous map $G:C\to K_h$ such that:
\begin{enumerate}[label=(\roman*)]
\item $h_\mu(f)\leq h(G(\mu))$ for every $\mu\in C$;
\item for every $\alpha\in K_h$ there exists $\mu_\alpha\in C$ satisfying
$
G(\mu_\alpha)=\alpha,
$ and $
h_{\mu_\alpha}(f)=h(\alpha).
$
\end{enumerate}
Then there exists $\Phi\in C(X,\mathbb R^m)$ such that
$$
R(\Phi;X,f)=K_h
$$
and
$$
\htop(f,L(\Phi,\alpha;X,f))=h(\alpha)
\qquad\text{for every }\alpha\in K_h.
$$
Moreover, $\Phi$ may be chosen so that
$
\int\Phi\,d\mu=G(\mu)
$
for every 
$
\mu\in C.
$
If, in addition, zero entropy ergodic measures are dense in
$\mathcal M_f(X)$, then for every
$N\in\mathbb N$ there exist
$
\Phi_1,\dots,\Phi_N\in C(X,\mathbb R^m)
$
such that, for every $j=1,\dots,N$,
$$
R(\Phi_j;X,f)=K_h
$$
and
$$
\htop(f,L(\Phi_j,\alpha;X,f))=h(\alpha)
\qquad\text{for every }\alpha\in K_h,
$$
and such that $\Phi_i$ and $\Phi_j$ are not cohomologous whenever
$i\neq j$.
\end{theorem}

\begin{proof}
	Put
	$
	F:=\overline{\operatorname{conv}}(C),
	$
	and let $\alpha_\ast$ be the unique maximizer of $h$. Since
	$C\subset\mathcal M_f^e(X)$ is compact,
	Proposition~\ref{prop-compact-ergodic-face} implies that $F$ is a closed
	face of $\mathcal M_f(X)$, and that $G:C\to K_h$ admits a unique continuous affine
	extension
	$$
	g:F\longrightarrow K_h.
	$$

	For each $\mu\in \mathcal M_f(X)$, define
	$$
	T(\mu):=\{\alpha\in K_h:h_\mu(f)\leq h(\alpha)\}.
	$$
	We verify the properties needed to apply Lazar's affine selection
	theorem.

	\medskip
	\noindent
	{\it Claim 1.} For every $\mu\in \mathcal M_f(X)$, the set $T(\mu)$ is a nonempty
	compact convex subset of $K_h$.

	\begin{proof}[Proof of Claim 1]
	Since
	$
	h_\mu(f)\leq\htop(f)=H=h(\alpha_\ast),
	$
	one has $\alpha_\ast\in T(\mu)$, so $T(\mu)$ is nonempty.
	Because $h$ is upper semicontinuous, $T(\mu)$ is a closed subset of the
	compact set $K_h$, and hence is compact. Finally, if
	$\alpha_1,\alpha_2\in T(\mu)$ and $\lambda\in[0,1]$, concavity of $h$
	gives
	$$
	\begin{aligned}
	h\bigl(\lambda\alpha_1+(1-\lambda)\alpha_2\bigr)
	\geq
	\lambda h(\alpha_1)+(1-\lambda)h(\alpha_2)
	\geq h_\mu(f).
	\end{aligned}
	$$
	Thus
	$\lambda\alpha_1+(1-\lambda)\alpha_2\in T(\mu)$, proving convexity.
	\end{proof}

	\medskip
	\noindent
	{\it Claim 2.} The set-valued map
	$
	T:\mathcal M_f(X)\to 2^{\mathbb R^m}
	$
	is lower semicontinuous.

	\begin{proof}[Proof of Claim 2]
	It suffices to show that for every relatively open set $U\subset K_h$, the
	set
	$
	\{\mu\in \mathcal M_f(X):T(\mu)\cap U\neq\varnothing\}
	$
	is open in $\mathcal M_f(X)$. Let $\mu_0$ belong to this set and choose
	$\alpha_0\in T(\mu_0)\cap U$.

	If $\alpha_0=\alpha_\ast$, then
	$
	h_\mu(f)\leq H=h(\alpha_\ast)
	$
	for every
	$
	\mu\in \mathcal M_f(X),
	$
	and hence $\alpha_\ast\in T(\mu)\cap U$ for every $\mu\in \mathcal M_f(X)$.
	Suppose that $\alpha_0\neq\alpha_\ast$. Since $U$ is relatively open in
	the convex set $K_h$, we may choose $\varepsilon>0$ sufficiently small
	that
	$$
	\beta=(1-\varepsilon)\alpha_0+\varepsilon\alpha_\ast\in U.
	$$
	The uniqueness of the maximizer gives $h(\alpha_0)<H$, and concavity
	yields
	$$
	\begin{aligned}
	h(\beta)
	\geq
	(1-\varepsilon)h(\alpha_0)+\varepsilon h(\alpha_\ast)
	>h(\alpha_0)
	\geq h_{\mu_0}(f).
	\end{aligned}
	$$
	By upper semicontinuity of the entropy map, there exists a weak$^*$
	neighborhood $V$ of $\mu_0$
	such that
	$
	h_\mu(f)<h(\beta)
	$
	for every $\mu\in V$. Thus $\beta\in T(\mu)\cap U$ for every
	$\mu\in V$. This proves the lower semicontinuity of $T$.
	\end{proof}

	\medskip
	\noindent
	{\it Claim 3.} For every $\mu,\nu\in \mathcal M_f(X)$ and $\lambda\in[0,1]$,
	$
	\lambda T(\mu)+(1-\lambda)T(\nu)
	\subset
	T\bigl(\lambda\mu+(1-\lambda)\nu\bigr).
	$

	\begin{proof}[Proof of Claim 3]
	Take $\alpha\in T(\mu)$ and $\beta\in T(\nu)$. By the affinity of
	entropy on invariant measures and the concavity of $h$,
	$$
	\begin{aligned}
	h_{\lambda\mu+(1-\lambda)\nu}(f)
	=
	\lambda h_\mu(f)+(1-\lambda)h_\nu(f)
	\leq
	\lambda h(\alpha)+(1-\lambda)h(\beta)
	\leq
	h\bigl(\lambda\alpha+(1-\lambda)\beta\bigr).
	\end{aligned}
	$$
	Hence
	$
	\lambda\alpha+(1-\lambda)\beta
	\in T(\lambda\mu+(1-\lambda)\nu)
	$.
	\end{proof}

	\medskip
	\noindent
	{\it Claim 4.} We have
	$
	g(\eta)\in T(\eta)
	$
	for every $\eta\in F$.

	\begin{proof}[Proof of Claim 4]
	Fix $\eta\in F$, and let $\xi_\eta\in\mathcal M(C)$ be its unique
	representing probability measure. Then
	$$
	\eta=\int_C\mu\,d\xi_\eta(\mu)
	\qquad\text{and}\qquad
	g(\eta)=\int_C G(\mu)\,d\xi_\eta(\mu).
	$$
	Using the affinity of entropy, condition~(i), and Jensen's inequality,
	we obtain
	$$
	\begin{aligned}
	h_\eta(f)
	=
	\int_C h_\mu(f)\,d\xi_\eta(\mu)
	\leq
	\int_C h(G(\mu))\,d\xi_\eta(\mu)
	\leq
	h\left(\int_C G(\mu)\,d\xi_\eta(\mu)\right)=
	h(g(\eta)).
	\end{aligned}
	$$
	Thus $g(\eta)\in T(\eta)$.
	\end{proof}

	Claims~1--4 allow us to apply Lemma~\ref{thm:lazar}. Therefore $g$
	extends to a continuous affine selection
	$$
	L:\mathcal M_f(X)\longrightarrow K_h
	$$
	such that
	$$
	L|_F=g,
	\qquad
	L(\mu)\in T(\mu)
	\quad\text{for every }\mu\in \mathcal M_f(X).
	$$
	By Lemma~\ref{lem:affine-representation}, there exists
	$\Phi\in C(X,\mathbb R^m)$ such that
	$$
	L(\mu)=\int\Phi\,d\mu
	\qquad\text{for every }\mu\in \mathcal M_f(X).
	$$
	Since $C\subset F$ and $g|_C=G$, this choice satisfies
	$
	\int\Phi\,d\mu=G(\mu)
	~\text{for every }\mu\in C.
	$

	We now prove that
	$
	R(\Phi;X,f)=K_h.
	$
	By Proposition~\ref{prop:shifts-facts},
	$$
	R(\Phi;X,f)
	=
	\left\{\int\Phi\,d\mu:\mu\in \mathcal M_f(X)\right\}
	=
	L(\mathcal M_f(X))
	\subset K_h.
	$$
	On the other hand, condition~(ii) implies that $G(C)=K_h$. Hence
	$$
	K_h=G(C)\subset g(F)\subset L(\mathcal M_f(X))=R(\Phi;X,f).
	$$
	Therefore
	$
	R(\Phi;X,f)=K_h.
	$

	Next fix $\alpha\in K_h$. If $\mu\in \mathcal M_f(X)$ satisfies
	$
	\int\Phi\,d\mu=\alpha,
	$
	then $L(\mu)=\alpha$. Since $L(\mu)\in T(\mu)$, it follows that
	$
	h_\mu(f)\leq h(\alpha).
	$
	Taking the supremum over all such $\mu$ and applying
	Proposition~\ref{prop:shifts-facts}, we obtain
	$$
	\htop(f,L(\Phi,\alpha;X,f))\leq h(\alpha).
	$$
	Conversely, condition~(ii) supplies $\mu_\alpha\in C$ such that
	$
	G(\mu_\alpha)=\alpha,
	$
	and
	$
	h_{\mu_\alpha}(f)=h(\alpha).
	$
	Because $L|_C=G$,
	$
	\int\Phi\,d\mu_\alpha=L(\mu_\alpha)=G(\mu_\alpha)=\alpha.
	$
	Another application of Proposition~\ref{prop:shifts-facts} gives
	$$
	\htop(f,L(\Phi,\alpha;X,f))
	\geq h_{\mu_\alpha}(f)
	=h(\alpha).
	$$
	Combining the two inequalities yields
	$$
	\htop(f,L(\Phi,\alpha;X,f))=h(\alpha)
	\qquad\text{for every }\alpha\in K_h.
	$$

	Assume now that zero entropy ergodic measures are dense in
	$\mathcal M_f(X)$, and fix $N\in\mathbb N$.
	Since zero entropy ergodic measures are dense in $\mathcal M_f(X)$,
	we may choose pairwise distinct measures
	$
	\zeta_1,\dots,\zeta_N\in\mathcal M_f^e(X)\setminus C
	$
	such that
	$
	h_{\zeta_j}(f)=0
	$
	for every $j=1,\dots,N$.
	Fix two distinct points $\alpha',\alpha''\in K_h$, and set
	$
	C_N:=C\cup\{\zeta_1,\dots,\zeta_N\}.
	$
	For each $k=1,\dots,N$, define $G^{(k)}:C_N\to K_h$ by
	$$
	G^{(k)}|_C=G
	$$
	and
	$$
	G^{(k)}(\zeta_j)
	=
	\begin{cases}
	\alpha'',&j=k,\\
	\alpha',&j\neq k.
	\end{cases}
	$$
	Since $C$ is compact and the added set is finite and disjoint from $C$,
	each $G^{(k)}$ is continuous on the compact set $C_N$.
	Let
	$
	F_N:=\overline{\operatorname{conv}}(C_N).
	$
	By Proposition~\ref{prop-compact-ergodic-face}, $F_N$ is a closed face
	of $\mathcal M_f(X)$, and each $G^{(k)}$ has a unique continuous affine extension
	$$
	g^{(k)}:F_N\longrightarrow K_h.
	$$
	For every $\mu\in C_N$,
	$$
	h_\mu(f)\leq h(G^{(k)}(\mu)).
	$$
	Consequently, if
	$\eta\in F_N$ and $\xi_\eta\in\mathcal M(C_N)$ is its representing
	probability measure, then the same Jensen argument as in Claim~4 gives
	$$
	\begin{aligned}
	h_\eta(f)
	=
	\int_{C_N}h_\mu(f)\,d\xi_\eta(\mu)\leq
	\int_{C_N}h(G^{(k)}(\mu))\,d\xi_\eta(\mu)
	\leq
	h\left(\int_{C_N}G^{(k)}(\mu)\,d\xi_\eta(\mu)\right)
	=
	h(g^{(k)}(\eta)).
	\end{aligned}
	$$
	Thus
	$
	g^{(k)}(\eta)\in T(\eta)
	$
	for every $\eta\in F_N$.

	By Claims~1--3 and Lemma~\ref{thm:lazar}, each $g^{(k)}$ extends to a
	continuous affine selection
	$$
	L^{(k)}:\mathcal M_f(X)\longrightarrow K_h
	$$
	satisfying
	$$
	L^{(k)}|_{F_N}=g^{(k)},
	\qquad
	L^{(k)}(\mu)\in T(\mu)
	\quad\text{for every }\mu\in \mathcal M_f(X).
	$$
	Lemma~\ref{lem:affine-representation} gives
	$\Phi_k\in C(X,\mathbb R^m)$ such that
	$$
	L^{(k)}(\mu)=\int\Phi_k\,d\mu
	\qquad\text{for every }\mu\in \mathcal M_f(X).
	$$
	Since $G^{(k)}|_C=G$ and condition~(ii) is unchanged, the preceding
	rotation-set and entropy-spectrum arguments apply verbatim to
	$\Phi_k$. Therefore, for every $k=1,\dots,N$,
	$$
	R(\Phi_k;X,f)=K_h
	$$
	and
	$$
	\htop(f,L(\Phi_k,\alpha;X,f))=h(\alpha)
	\qquad\text{for every }\alpha\in K_h.
	$$
	Moreover,
	$
	\int\Phi_k\,d\mu=G^{(k)}(\mu)
	~\text{for every }\mu\in C_N.
	$

	We finally prove that $\Phi_1,\dots,\Phi_N$ are pairwise
	non-cohomologous. Fix $i\neq j$ and any $\rho\in C$. Since all the maps
	$G^{(k)}$ agree with $G$ on $C$,
	$$
	\int(\Phi_i-\Phi_j)\,d\rho=0,
	$$
	whereas
	$$
	\int(\Phi_i-\Phi_j)\,d\zeta_i
	=
	\alpha''-\alpha'
	\neq0.
	$$
	If $\Phi_i-\Phi_j=c+u-u\circ f$ for some $c\in\mathbb R^m$ and
	$u\in C(X,\mathbb R^m)$, then its integral would equal $c$ for every
	invariant measure, contradicting the preceding two identities.
	Therefore the potentials are pairwise non-cohomologous. This completes
	the proof.
\end{proof}

\section{Complete multiparameter spectrum realization on transitive SFTs}\label{sec:multiparameter-realization}\label{mp:sec:multiparameter-realization}

The proof proceeds in four steps.  First, we realize prescribed subcritical
entropy profiles within transitive SFTs.  Second, we construct pairwise
separated SFT layers whose topological entropies converge to that of the
ambient system.  Third, we embed successive shells of the target hypograph
into these layers, thereby obtaining a compact calibrated pair.  Finally, we
apply the common realization principle to this pair to obtain the vector
potentials asserted in Theorem~\ref{mp:thm:main-realization}.

\subsection{Symbolic preliminaries}\label{mp:sec:preliminaries}

Let $\mathcal A$ be a finite alphabet.  The \emph{full two-sided shift} over
$\mathcal A$ is $(\Sigma_{\mathcal A},\sigma)$, where
$\Sigma_{\mathcal A}=\mathcal A^{\bbZ}$ and the shift map
$\sigma:\Sigma_{\mathcal A}\to\Sigma_{\mathcal A}$ is defined, for
$x=(x_n)_{n\in\bbZ}\in\Sigma_{\mathcal A}$, by
$$
(\sigma x)_n=x_{n+1}
\qquad\text{for every }n\in\bbZ.
$$
A \emph{two-sided subshift} is a closed $\sigma$-invariant subset $X\subset\Sigma_{\mathcal A}$.  Its language is denoted by
$
\mathcal L(X)=\bigcup_{n\ge1}\mathcal L_n(X),
$
where $\mathcal L_n(X)$ is the set of words of length $n$ occurring in points of $X$.  If
$w=w_0\cdots w_{n-1}\in\mathcal L_n(X)$, its coordinate-zero cylinder is
$
[w]=\{x\in X:x_0\cdots x_{n-1}=w\}.
$

A \emph{shift of finite type} (SFT) is specified by a finite zero-one
transition matrix $A=(A_{ij})_{1\le i,j\le s}$ as
$$
X_A
=
\left\{x=(x_n)_{n\in\bbZ}\in\{1,\ldots,s\}^{\bbZ}:
A_{x_nx_{n+1}}=1\text{ for every }n\in\bbZ
\right\}.
$$
Here $A$ is \emph{irreducible} if, for every pair of states $i,j$, one has $(A^n)_{ij}>0$ for some $n\ge1$, and it is \emph{primitive} if $A^N$ has all entries positive for some $N\ge1$.

A subshift is \emph{topologically transitive} if for every pair of nonempty open sets $U,V\subset X$, one has
$\sigma^n(U)\cap V\ne\varnothing$ for some $n\ge0$, and it is \emph{topologically mixing} if this holds for every sufficiently large $n$.  The SFT $X_A$ is transitive if and only if $A$ is irreducible, and it is mixing if and only if $A$ is primitive.  If $A$ is irreducible, the \emph{period} of $X_A$ is
$
p=\gcd\{n\ge1:(A^n)_{ii}>0\},
$
and this number is independent of the state $i$. Moreover, $p=1$ exactly in
the mixing case. A transitive SFT is called \emph{nontrivial} if it is not a
single periodic orbit, equivalently if its topological entropy is positive.

Given two subshifts $X$ and $Y$, a map $\pi:X\to Y$ is
\emph{shift-commuting} if
$$
\pi(\sigma x)=\sigma(\pi(x))
\qquad\text{for every }x\in X.
$$
A \emph{topological conjugacy} between two subshifts is a shift-commuting
homeomorphism. A \emph{factor map} is a continuous surjective
shift-commuting map, and an \emph{embedding} is an injective continuous
shift-commuting map.  A \emph{sofic shift} is a factor of an SFT.  A
dynamical system $(X,f)$ is \emph{minimal} if
$
\overline{\{f^n(x):n\geq0\}}=X
$
for every $x\in X$.  A point $z\in\mathcal A^{\bbZ}$ is a \emph{Toeplitz sequence} if, for every coordinate $i\in\bbZ$, there is an integer $p_i\ge1$ such that
$
z_{i+kp_i}=z_i
$
for every 
$
k\in\bbZ.
$
The orbit closure of a Toeplitz sequence is called a \emph{Toeplitz
subshift}. Every Toeplitz subshift is minimal: indeed, each finite block in
a Toeplitz sequence recurs periodically and hence with bounded gaps.  The
preceding SFT and sofic conventions agree with the standard terminology in
symbolic dynamics; see \cite[Chapters~1--4]{LindMarcus2021}.

\subsection{Prescribed entropy embeddings in transitive SFTs}\label{mp:sec:SFT-entropy-sections}

This part constructs the prescribed entropy embeddings used on the
high entropy SFT layers by combining a full shift with pairwise
disjoint iterates, entropy realization on Toeplitz subshifts,
Krieger's embedding theorem, and cyclic transfer of invariant measures and
entropy.

\begin{lemma}\label{mp:prop:disjoint-full-shift-iterates}\cite[Lemma~7.9]{AvilaCrovisierWilkinson2021}
Let $(Y,\sigma)$ be a nontrivial transitive two-sided SFT. For every
$0<\eta<\htop(\sigma,Y)$, there exist integers $k\ge1$ and $q\ge2$, together with a compact set $\Lambda\subset Y$, such that:
\begin{enumerate}[label=(\roman*)]
\item $\sigma^k(\Lambda)=\Lambda$;
\item $(\Lambda,\sigma^k)$ is topologically conjugate to the two-sided full shift on $q$ symbols;
\item the sets
$
\Lambda,\sigma(\Lambda),\dots,\sigma^{k-1}(\Lambda)
$
are pairwise disjoint;
\item
$
\frac1k\log q>\eta.
$
\end{enumerate}
\end{lemma}

\begin{lemma}\label{mp:lem:cyclic-measure-transfer}
Let $f:Z_0\to Z_0$ be a homeomorphism of a compact metric space.  Let $\Omega\subset Z_0$ be compact with $f^k(\Omega)=\Omega$, and suppose that
$
\Omega,f(\Omega),\dots,f^{k-1}(\Omega)
$
are pairwise disjoint.  Put
$
Z=\bigcup_{i=0}^{k-1}f^i(\Omega).
$
Then $Z$ is compact and $f$-invariant, and the averaging map
$$
\mathcal A:\mathcal M_{f^k}(\Omega)\longrightarrow \mathcal M_f(Z),
\qquad
\mathcal A(\nu)=\frac1k\sum_{i=0}^{k-1}(f^i)_\ast\nu,
$$
is an affine homeomorphism.  Its inverse is $\mu\mapsto k\mu|_\Omega$.  It preserves ergodicity and satisfies
$
h_{\mathcal A(\nu)}(f)=\frac1k h_\nu(f^k).
$
If $(\Omega,f^k)$ is minimal, then $(Z,f)$ is minimal.
\end{lemma}

\begin{proof}
The sets $\Omega,f(\Omega),\ldots,f^{k-1}(\Omega)$ form a clopen cyclic decomposition of the compact invariant set $Z$.  Standard measure-transfer facts for such a decomposition imply that $\mathcal A$ and $\mu\mapsto k\mu|_\Omega$ are inverse homeomorphisms, preserve ergodicity, and satisfy
$
h_{\mathcal A(\nu)}(f)=\frac1k h_\nu(f^k).
$
See, for example, \cite[Proposition~23.17]{DenkerGrillenbergerSigmund1976}. The definition of $\mathcal A$ also shows that it is affine.
If $(\Omega,f^k)$ is minimal, then $f^k$ is minimal on every cyclic
component $f^i(\Omega)$. Hence, for each $z\in Z$, the forward $f$-orbit of
$z$ meets every cyclic component in a dense set. Therefore
$
\overline{\{f^n(z):n\geq0\}}=Z,
$
which proves that $(Z,f)$ is minimal.
\end{proof}

\begin{lemma}\cite[Theorem~1]{DownarowiczSerafin2003}\label{mp:lem:toeplitz-entropy-realization}
Let $\mathcal K$ be a Choquet simplex, and let
$
e:\mathcal K\to[0,\infty)
$
be affine and upper semicontinuous.  Then there exist a Toeplitz subshift $(W,\sigma)$ and an affine homeomorphism
$
\varphi:\mathcal K\longrightarrow\mathcal M_\sigma(W)
$
such that
$
h_{\varphi(\kappa)}(\sigma)=e(\kappa)
$
for every 
$
\kappa\in\mathcal K.
$
\end{lemma}

For a shift space $X$, let $q_n(X)$ denote the number of points of least period $n$.

\begin{lemma}\cite[Corollary~10.1.9]{LindMarcus2021}\label{mp:lem:krieger-embedding}
Let $X$ be a subshift and $Y$ a mixing two-sided SFT.  Then there exists an embedding $\iota:X\hookrightarrow Y$ with $\iota(X)\ne Y$ if and only if
$
\htop(\sigma,X)<\htop(\sigma,Y)
$
and
$
q_n(X)\le q_n(Y)
$
for every 
$
n\ge1.
$
\end{lemma}

\begin{theorem}\label{mp:thm:subcritical-entropy-sections}
Let $(Y,\sigma)$ be a nontrivial transitive two-sided SFT.  Let $E$ be a nonempty compact metric space, and let
$
e:E\to[0,\htop(\sigma,Y))
$
be upper semicontinuous.  Then there exists a topological embedding (that
is, a homeomorphism onto its image)
$
s:E\longrightarrow \mathcal M_\sigma^e(Y)
$
such that
$$
h_{s(x)}(\sigma)=e(x)
\qquad\text{for every }x\in E.
$$
\end{theorem}

\begin{proof}
Since $E$ is compact, upper semicontinuity of $e$ and the pointwise strict inequality $e(x)<\htop(\sigma,Y)$ imply
$
\max_E e<\htop(\sigma,Y).
$
Let $\mathcal M(E)$ be the Bauer simplex of Borel probability measures on
$E$, and define
$$
\widehat e(\lambda)=\int_E e(x)\,d\lambda(x)
\qquad\text{for every }\lambda\in\mathcal M(E).
$$
The map $\widehat e$ is affine and upper semicontinuous. Indeed, if
$\lambda_j\to\lambda$ weak$^*$ in $\mathcal M(E)$, then the Portmanteau
theorem gives
$
\limsup_{j\to\infty}\widehat e(\lambda_j)
\leq \widehat e(\lambda).
$
Moreover,
$
\max_{\lambda\in\mathcal M(E)}\widehat e(\lambda)
=\max_E e
<\htop(\sigma,Y).
$
Choose $\eta$ so that
$
\max_E e<\eta<\htop(\sigma,Y).
$
By Lemma~\ref{mp:prop:disjoint-full-shift-iterates}, there are integers $k\ge1$, $q\ge2$, and a compact set $\Lambda\subset Y$ such that $(\Lambda,\sigma^k)$ is conjugate to the full two-sided $q$-shift, the sets
$
\Lambda,\sigma(\Lambda),\dots,\sigma^{k-1}(\Lambda)
$
are pairwise disjoint, and
$
\frac1k\log q>\eta.
$

Apply Lemma~\ref{mp:lem:toeplitz-entropy-realization} to $k\widehat e$ on $\mathcal M(E)$.  This gives a Toeplitz subshift $(W,\sigma)$ and an affine homeomorphism
$$
\varphi:\mathcal M(E)\longrightarrow\mathcal M_\sigma(W)
$$
such that
$$
h_{\varphi(\lambda)}(\sigma)=k\widehat e(\lambda)
\qquad\text{for every }\lambda\in\mathcal M(E).
$$
By the variational principle,
$
\htop(\sigma,W)=k\max_E e<k\eta<\log q.
$
Since $W$ is a Toeplitz subshift, it is minimal. Consequently, it either has
no periodic points or is a finite periodic orbit.  In the latter case, the
full $q$-shift contains a periodic orbit of the same period.  Hence
Lemma~\ref{mp:lem:krieger-embedding} gives an embedding
$
\iota:W\hookrightarrow\Sigma_q.
$

Let
$
\chi:\Sigma_q\longrightarrow\Lambda
$
be a conjugacy satisfying $\chi\circ\sigma=\sigma^k\circ\chi$.  Set
$
\Omega=\chi(\iota(W))
$
and
$
Z=\bigcup_{i=0}^{k-1}\sigma^i(\Omega).
$
The sets $\Omega,\sigma(\Omega),\ldots,\sigma^{k-1}(\Omega)$ are pairwise disjoint.  Lemma~\ref{mp:lem:cyclic-measure-transfer} gives an affine homeomorphism
$$
\mathcal A:\mathcal M_{\sigma^k}(\Omega)\longrightarrow\mathcal M_\sigma(Z).
$$

Let $\psi=\chi\circ\iota:W\to\Omega$. This is a conjugacy from
$(W,\sigma)$ to $(\Omega,\sigma^k)$. Denote the induced pushforward map on
invariant measures by
$$
\psi_\ast:\mathcal M_\sigma(W)\longrightarrow
\mathcal M_{\sigma^k}(\Omega),
\qquad
\psi_\ast\nu:=\nu\circ\psi^{-1}.
$$
Define
$$
\Psi=\mathcal A\circ\psi_\ast\circ\varphi:
\mathcal M(E)\longrightarrow\mathcal M_\sigma(Z).
$$
Then $\Psi$ is an affine homeomorphism and
$$
h_{\Psi(\lambda)}(\sigma)=\widehat e(\lambda)
\qquad\text{for every }\lambda\in\mathcal M(E).
$$
The extreme points of $\mathcal M(E)$ are precisely the Dirac measures.
Indeed, every $\delta_x$ is extreme. If $\lambda\in\mathcal M(E)$ is not a
Dirac measure, there exists a Borel set $B\subset E$ with
$0<\lambda(B)<1$, and
$$
\lambda
=
\lambda(B)\frac{\lambda|_B}{\lambda(B)}
+
\bigl(1-\lambda(B)\bigr)
\frac{\lambda|_{E\setminus B}}{1-\lambda(B)}
$$
is a nontrivial convex decomposition. Thus
$
\operatorname{ext}\mathcal M(E)=\{\delta_x:x\in E\}
$.
Since affine homeomorphisms preserve extreme points, the map
$
s(x):=\Psi(\delta_x)
$
is a continuous embedding from $E$ into
$\mathcal M_\sigma^e(Z)\subset\mathcal M_\sigma^e(Y)$. Finally,
$
h_{s(x)}(\sigma)=\widehat e(\delta_x)=e(x)
$
for every 
$
x\in E.
$
\end{proof}

\subsection{High entropy separated SFT layers}\label{mp:sec:local-layers}

Let $(X,\sigma)$ be a fixed nontrivial transitive two-sided SFT, let
$m_{\max}$ be its unique measure of maximal entropy, and fix a compatible
metric $d$ on $\mathcal M_\sigma(X)$, with induced Hausdorff distance $d_H$.
We construct transitive SFTs $\Lambda_n\subset X$ such that
$$
\htop(\sigma,\Lambda_n)\longrightarrow\htop(\sigma,X),
\qquad
d_H\bigl(\mathcal M_\sigma(\Lambda_n),\{m_{\max}\}\bigr)\longrightarrow0,
$$
and the sets $\mathcal M_\sigma(\Lambda_n)$ are pairwise disjoint.

For every periodic point $x\in X$ with period $n$, define the associated periodic measure by
$
\mu_x:=\frac1n\sum_{k=0}^{n-1}\delta_{\sigma^k(x)}.
$
The next lemma is used not only in the construction of the high entropy separated SFT layers below, but also in the proof of Theorem~\ref{thm:not-continuous}.

\begin{lemma}\label{lem:periodic-measure-monotone-path-transitive-sft}
	Let $(X,\sigma)$ be a nontrivial transitive two-sided SFT. Then for any periodic measure $\mu_x\in \mathcal M_\sigma(X)$, there exists a continuous path $\{\mu_t\}_{t\in[0,1]}\subset \mathcal M_\sigma^e(X)$ such that $\mu_0=\mu_x$, the map
	$
	t\mapsto h_{\mu_t}(\sigma)
	$
	is continuous and strictly increasing on $[0,1]$, $\operatorname{Supp}(\mu_t)=X$ for every $t\in(0,1]$, and $\mu_1$ is the unique measure of maximal entropy.
\end{lemma}

\begin{proof}
	We first prove the lemma when $(X,\sigma)$ is mixing.
	Let $\operatorname{Supp}(\mu_x)$ be the periodic orbit supporting $\mu_x$. Fix a compatible metric $d_X$ on $X$, and define
	$$
	\varphi(y):=-d_X(y,\operatorname{Supp}(\mu_x))
	\qquad\text{for every }y\in X.
	$$
	It is immediate that
	$
	\mathcal M_{\max}(\varphi)=\{\mu_x\}.
	$
	Since $\operatorname{Supp}(\mu_x)$ is closed, $\varphi$ is Lipschitz. 
	
	For $s\ge 0$, write
	$
	P(s):=P_{\sigma}(s\varphi).
	$
	For a mixing SFT, every Lipschitz potential $\phi$ has a unique equilibrium
	state $\mu_\phi$ with full support, that is,
	$\operatorname{Supp}(\mu_\phi)=X$; see \cite{Bowen2008}. For each $s\ge 0$,
	let $\nu_s$ be the unique equilibrium state of the Lipschitz potential
	$s\varphi$. Then $\operatorname{Supp}(\nu_s)=X$. Since the entropy map of
	$(X,\sigma)$ is upper semicontinuous, Proposition~\ref{prop:pressure}(4)
	implies that the map
	$
	s\mapsto \nu_s
	$
	is weak$^*$ continuous on $[0,\infty)$. Also, by
	Proposition~\ref{prop:pressure}, for every $\psi\in C(X)$ one has
	$$
	\left.\frac{d}{dt}P_{\sigma}(s\varphi+t\psi)\right|_{t=0}
	=
	\int \psi\,d\nu_s.
	$$
	In particular, taking $\psi=\varphi$, we obtain
	$
	P'(s)=\int \varphi\,d\nu_s.
	$
	Furthermore, by \cite{Ruelle1977}, the pressure function $P(s)$ is real analytic in $s$.
	
	Since $\nu_s$ is the equilibrium state of $s\varphi$, we have
	$
	P(s)=h_{\nu_s}(\sigma)+s\int \varphi\,d\nu_s,
	$
	hence
	$
	h_{\nu_s}(\sigma)=P(s)-sP'(s).
	$
	Differentiating gives
	$$
	\frac{d}{ds}h_{\nu_s}(\sigma)=-sP''(s).
	$$
	
	We claim that $P''(s)>0$ for every $s>0$. Recall that $\varphi$ is said to be cohomologous to a constant if there exist a continuous function $u$ and a constant $c$ such that
	$
	\varphi=c+u-u\circ \sigma.
	$
	In that case, $\varphi$ has the same average $c$ on every periodic orbit. In our situation, $\varphi\equiv 0$ on $\operatorname{Supp}(\mu_x)$, while $\varphi<0$ on every periodic orbit disjoint from $\operatorname{Supp}(\mu_x)$. Since $X$ is a nontrivial mixing SFT, it contains a periodic orbit different from $\operatorname{Supp}(\mu_x)$. Therefore $\varphi$ is not cohomologous to a constant. It then follows from \cite[Proposition 4.12]{ParryPollicott1990} that
	$
	P''(s)>0
	$
	for all
	$
	s>0.
	$
	Consequently,
	$$
	\frac{d}{ds}h_{\nu_s}(\sigma)<0 \qquad \text{for every } s>0.
	$$
	Thus
	$
	s\mapsto h_{\nu_s}(\sigma)
	$
	is strictly decreasing on $(0,\infty)$.
	Moreover, for every $s>0$,
	the strict convexity of $P$ gives
	$
	P(0)>P(s)-sP'(s).
	$
	Since
	$
	h_{\nu_0}(\sigma)=P(0)
	$
	and
	$
	h_{\nu_s}(\sigma)=P(s)-sP'(s),
	$
	we obtain
	$
	h_{\nu_0}(\sigma)>h_{\nu_s}(\sigma)
	$
	for every $s>0$. Hence
	$
	s\mapsto h_{\nu_s}(\sigma)
	$
	is strictly decreasing on $[0,\infty)$.
	
	Next we show that
	$
	\nu_s \xrightarrow[s\to\infty]{w^*} \mu_x.
	$
	Since $\nu_s$ is an equilibrium state for $s\varphi$, we have
	\begin{equation}\label{equ-AA}
		P(s)=h_{\nu_s}(\sigma)+s\int \varphi\,d\nu_s.
	\end{equation}
	On the other hand, because $\mu_x$ is $\sigma$-invariant and $\int\varphi\,d\mu_x=0$, we get
	\begin{equation}\label{equ-AB}
		P(s)\ge h_{\mu_x}(\sigma)+s\int \varphi\,d\mu_x=0.
	\end{equation}
	Also $h_{\nu_s}(\sigma)\le \htop(\sigma)$, hence
	$
	0\le P(s)\le \htop(\sigma)+s\int \varphi\,d\nu_s.
	$
	Therefore
	$
	\int \varphi\,d\nu_s \ge -\frac{\htop(\sigma)}{s}.
	$
	Since always $\int \varphi\,d\nu_s\le 0$, we conclude that
	$
	\int \varphi\,d\nu_s \to 0 
	$
	as
	$s\to\infty.
	$
	If $s_k\to\infty$ and $\nu_{s_k}\to \nu$ weak$^*$ along a subsequence, then by continuity of $\varphi$,
	$
	\int \varphi\,d\nu=\lim_{k\to\infty}\int \varphi\,d\nu_{s_k}=0.
	$
	By the uniqueness of the maximizing measure of $\varphi$, it follows that $\nu=\mu_x$. Therefore every weak$^*$ accumulation point of the family $\{\nu_s\}_{s\ge 0}$ is equal to $\mu_x$. Since $\mathcal M(X)$ is weak$^*$ compact, we conclude that
	$
	\nu_s \xrightarrow[s\to\infty]{w^*} \mu_x.
	$
	
	Now define
	$$
	\mu_0:=\mu_x,
	\qquad
	\mu_t:=\nu_{\frac{1-t}{t}} \quad \text{for } t\in(0,1].
	$$
	Then $\mu_t\in \mathcal M_\sigma^e(X)$ for every $t\in[0,1]$. The path $t\mapsto \mu_t$ is weak$^*$ continuous on $(0,1]$ by the weak$^*$ continuity of $s\mapsto \nu_s$, and it is continuous at $t=0$ because $(1-t)/t\to\infty$ as $t\downarrow 0$ and $\nu_s\to\mu_x$ as $s\to\infty$.
	
	Finally, since $t\mapsto (1-t)/t$ is strictly decreasing on $(0,1]$
	and $s\mapsto h_{\nu_s}(\sigma)$ is strictly decreasing on $[0,\infty)$,
	the map
	$$
	t\mapsto h_{\mu_t}(\sigma)
	=
	h_{\nu_{\frac{1-t}{t}}}(\sigma)
	$$
	is strictly increasing on $(0,1]$. Since $\varphi\leq 0$ on $X$, $\varphi<0$ on some nonempty open set,
	and $\nu_s$ has full support, we have
	$\int \varphi\,d\nu_s<0$. Hence
	$-s\int \varphi\,d\nu_s>0$ for every $s>0$. By (\ref{equ-AA}) and (\ref{equ-AB}),
	we have
	$
	h_{\nu_s}(\sigma)\geq -s\int \varphi\,d\nu_s>0
	$
	for every $s>0$. Then
	$
	h_{\mu_0}(\sigma)=h_{\mu_x}(\sigma)=0
	$
	and $h_{\mu_t}(\sigma)>0$ for every $t>0$. Thus
	$t\mapsto h_{\mu_t}(\sigma)$ is strictly increasing on all of $[0,1]$.
	Moreover,
	$
	\operatorname{Supp}(\mu_t)
	=
	\operatorname{Supp}\left(\nu_{\frac{1-t}{t}}\right)
	=
	X
	$
	for every $t\in(0,1]$. The identity $h_{\nu_s}(\sigma)=P(s)-sP'(s)$ and the real analyticity of $P$ show that $t\mapsto h_{\mu_t}(\sigma)$ is continuous on $(0,1]$. At $t=0$, upper semicontinuity of the entropy map and $\mu_t\to\mu_x$ give
	$$
	0\le\limsup_{t\downarrow0}h_{\mu_t}(\sigma)\le h_{\mu_x}(\sigma)=0.
	$$
	Thus the entropy path is continuous on $[0,1]$. Since $\mu_1=\nu_0$ and $\nu_0$ is the unique equilibrium state of the zero potential, $\mu_1$ is the unique measure of maximal entropy.
	This proves the mixing case.

		Suppose now that $X$ is not mixing, and let $p\ge2$ be its period.
	The cyclic decomposition gives pairwise disjoint clopen sets
	$$
	X=X_0\sqcup X_1\sqcup\cdots\sqcup X_{p-1},
	\qquad
		\sigma(X_i)=X_{(i+1)\bmod p},
	$$
	such that $(X_i,\sigma^p|_{X_i})$ is a mixing two-sided SFT for every $i$ (see, for example,
	\cite[Proposition 4.5.6]{LindMarcus2021}).
	After replacing $x$ by a point on the same periodic orbit, we may assume
	$x\in X_0$. If $n$ is the period of $x$, then $p$ divides $n$, and
	$$
	\rho_x
	:=
	\frac{p}{n}\sum_{j=0}^{n/p-1}\delta_{\sigma^{pj}(x)}
	$$
	is the periodic measure of $x$ for the mixing system
	$(X_0,\sigma^p|_{X_0})$. By the mixing case, there is a continuous path
	$
	[0,1]\ni t\longmapsto\rho_t
	\in\mathcal M_{\sigma^p}^e(X_0)
	$
	starting at $\rho_x$, with $t\mapsto h_{\rho_t}(\sigma^p)$ continuous and
	strictly increasing and $\operatorname{Supp}(\rho_t)=X_0$ for every $t>0$.

	Consider the cyclic averaging map
	$
	\mathcal A:\mathcal M_{\sigma^p}(X_0)\longrightarrow
	\mathcal M_\sigma(X),
	~
	\mathcal A(\rho)=\frac1p\sum_{i=0}^{p-1}(\sigma^i)_*\rho.
	$
	Lemma~\ref{mp:lem:cyclic-measure-transfer}, applied with $\Omega=X_0$ and
	$k=p$, shows that $\mathcal A$ is an affine homeomorphism, preserves
	ergodicity, and satisfies
	$
	h_{\mathcal A(\rho)}(\sigma)=\frac1p h_\rho(\sigma^p).
	$
	Define $\mu_t=\mathcal A(\rho_t)$. Then $t\mapsto\mu_t$ is a continuous
	path in $\mathcal M_\sigma^e(X)$ and
	$$
	\mu_0
	=
	\frac1p\sum_{i=0}^{p-1}(\sigma^i)_*\rho_x
	=
	\frac1n\sum_{k=0}^{n-1}\delta_{\sigma^k(x)}
	=
	\mu_x.
	$$
	The entropy-scaling identity shows that
	$t\mapsto h_{\mu_t}(\sigma)$ is continuous and strictly increasing. Since $\rho_1$ is the
	unique measure of maximal entropy for $(X_0,\sigma^p|_{X_0})$, the
	entropy-scaling identity and the bijectivity of $\mathcal A$ show that
	$\mu_1=\mathcal A(\rho_1)$ is the unique measure of maximal entropy for
	$(X,\sigma)$. Finally, for $t>0$,
	$$
	\operatorname{Supp}(\mu_t)
	=
	\bigcup_{i=0}^{p-1}\sigma^i(\operatorname{Supp}(\rho_t))
	=
	\bigcup_{i=0}^{p-1}X_i
	=
	X.
	$$
	This completes the proof.
\end{proof}

The following local horseshoe-type result, used in the layer construction, was established in \cite{DongHouTian2022}.  For completeness, the
\hyperref[app:proof-local-sft-approximation]{appendix} gives a slightly
different proof.

\begin{lemma}\label{mp:lem:one-measure-sft-approximation}
Let $\mu\in \mathcal M_\sigma^e(X)$, let $\eta>0$, and let $U$ be a weak$^*$ neighborhood of $\mu$ in $\mathcal M_\sigma(X)$.  Then there exists a compact invariant set $\Lambda\subset X$ such that:
\begin{enumerate}[label=(\roman*)]
\item $(\Lambda,\sigma|_\Lambda)$ is a transitive two-sided SFT;
\item
$
\htop(\sigma,\Lambda)>h_\mu(\sigma)-\eta;
$
\item
$
\mathcal M_\sigma(\Lambda)\subset U.
$
\end{enumerate}
\end{lemma}

We now apply this local approximation result to construct pairwise separated
high entropy layers.

\begin{proposition}\label{mp:prop:separated-layers}
Put
$
\delta_n=2^{-n}\htop(\sigma,X)
$
for every $n\ge1$. There exist compact invariant sets
$\Lambda_n\subset X$, $n\ge1$, with the following properties:
\begin{enumerate}[label=(\roman*)]
\item $(\Lambda_n,\sigma|_{\Lambda_n})$ is a transitive two-sided SFT;
\item
$
\htop(\sigma,\Lambda_n)>\htop(\sigma,X)-\frac{\delta_n}{2};
$
\item
$
d_H(\mathcal M_\sigma(\Lambda_n),\{m_{\max}\})<2^{-n};
$
\item the sets $\mathcal M_\sigma(\Lambda_n)$ are pairwise disjoint.
\end{enumerate}
\end{proposition}

\begin{proof}
Set
$
\gamma_n=\htop(\sigma,X)-\frac{\delta_n}{2}
$ and $
\zeta_n=2^{-n}.
$
Choose a periodic measure $\mu_x$ and let $(\mu_t)_{0\le t\le1}$ be the path from Lemma~\ref{lem:periodic-measure-monotone-path-transitive-sft}.  We choose inductively parameters $t_n\in(0,1)$, ergodic measures
$
\nu_n=\mu_{t_n},
$
and open weak$^*$ neighborhoods $V_n$ of $\nu_n$ such that:
\begin{align}
&h_{\nu_n}(\sigma)>\frac{\htop(\sigma,X)+\gamma_n}{2},\label{mp:eq:center-entropy}\\
&\overline V_n\cap\overline V_i=\varnothing
\quad\text{for every }i<n,\label{mp:eq:V-disjoint}\\
&m_{\max}\notin\overline V_n,\label{mp:eq:m-not-V}\\
&d(\nu_n,m_{\max})+\operatorname{diam}(V_n)<\zeta_n.\label{mp:eq:V-small}
\end{align}

Suppose $V_1,\dots,V_{n-1}$ have been chosen.  Put
$
D_{n-1}=\bigcup_{i<n}\overline V_i.
$
If $n=1$, set $r_n=1$.  Otherwise, $D_{n-1}$ is compact and does not contain $m_{\max}$, so
$$
r_n=\frac12\operatorname{dist}(m_{\max},D_{n-1})>0.
$$
By continuity of the path and of its entropy in Lemma~\ref{lem:periodic-measure-monotone-path-transitive-sft}, as $t\uparrow1$ one has both $\mu_t\to m_{\max}$ and $h_{\mu_t}(\sigma)\to\htop(\sigma,X)$.  Hence we may choose $t_n<1$, with $t_n>(1+t_{n-1})/2$ when $n>1$, such that
$$
d(\mu_{t_n},m_{\max})<\min\{r_n,\zeta_n/4\}
,\quad
h_{\mu_{t_n}}(\sigma)>\frac{\htop(\sigma,X)+\gamma_n}{2}.
$$
The first inequality places $\nu_n$ outside $D_{n-1}$.  Since $\nu_n\ne m_{\max}$, we can choose an open ball $V_n$ centered at $\nu_n$ so small that its closure is disjoint from $D_{n-1}$, does not contain $m_{\max}$, and satisfies \eqref{mp:eq:V-small}.  This completes the induction.

For each $n$, choose $\eta_n>0$ with
$
0<\eta_n<h_{\nu_n}(\sigma)-\gamma_n.
$
Since $\nu_n$ is ergodic, apply Lemma~\ref{mp:lem:one-measure-sft-approximation} to $\nu_n$, $\eta_n$, and the neighborhood $V_n$.  We obtain a transitive SFT $\Lambda_n\subset X$ such that
$
\htop(\sigma,\Lambda_n)
>
h_{\nu_n}(\sigma)-\eta_n
>
\gamma_n
$
and
$
\mathcal M_\sigma(\Lambda_n)\subset V_n.
$
The neighborhoods $V_n$ are pairwise disjoint, so the compact measure sets $\mathcal M_\sigma(\Lambda_n)$ are pairwise disjoint.
Finally, for every $\mu\in \mathcal M_\sigma(\Lambda_n)\subset V_n$,
$$
d(\mu,m_{\max})
\le
 d(\mu,\nu_n)+d(\nu_n,m_{\max})
\le
 \operatorname{diam}(V_n)+d(\nu_n,m_{\max})
<
\zeta_n.
$$
Thus
$
d_H(\mathcal M_\sigma(\Lambda_n),\{m_{\max}\})<\zeta_n.
$
\end{proof}

\subsection{Hypograph shells and the calibrated pair}\label{mp:sec:calibrated-pair}

The separated layers now allow the entire hypograph of an admissible target
to be encoded by ergodic measures. These layers can accumulate only at the
unique measure of maximal entropy, corresponding to the unique maximizer of
the target spectrum.

\begin{proposition}\label{mp:prop:compact-calibrated-pair}
Let $(X,\sigma)$ be a nontrivial transitive two-sided SFT, and let
$h\in C_{ms}^{m,\htop(\sigma,X)}$. Then there exist a compact set
$
C\subset\mathcal M_\sigma^e(X)
$
and a continuous map $G:C\to K_h$ such that
$$
h_\mu(\sigma)\leq h(G(\mu))
\qquad\text{for every }\mu\in C,
$$
and, for every $\alpha\in K_h$, there exists $\mu_\alpha\in C$ satisfying
$
G(\mu_\alpha)=\alpha,
$
$
h_{\mu_\alpha}(\sigma)=h(\alpha).
$
\end{proposition}

\begin{proof}
Put $K=K_h$, and, as in Proposition~\ref{mp:prop:separated-layers}, set
$
\delta_n=2^{-n}\htop(\sigma,X)
$
for every $n\ge0$.  Denote the unique maximizer of $h$ by $\alpha_\ast$. Define the hypograph
$$
\mathcal E
=
\{(\alpha,r)\in K\times[0,\htop(\sigma,X)]:0\le r\le h(\alpha)\}.
$$
To see that $\mathcal E$ is closed, suppose that
$(\alpha_j,r_j)\in\mathcal E$ and
$(\alpha_j,r_j)\to(\alpha,r)$.  Since $0\le r_j\le h(\alpha_j)$, upper
semicontinuity gives
$
0\le r
\le\limsup_{j\to\infty}h(\alpha_j)
\le h(\alpha).
$
Thus $(\alpha,r)\in\mathcal E$.  Hence $\mathcal E$ is a closed subset of
the compact set $K\times[0,\htop(\sigma,X)]$, and is therefore compact.
Concavity of $h$ also makes
$\mathcal E$ convex.  For $n\ge1$, define the compact shell
$$
\mathcal E_n
=
\{(\alpha,r)\in\mathcal E:
\delta_n\le \htop(\sigma,X)-r\le\delta_{n-1}\}.
$$
The point $(\alpha_\ast,\htop(\sigma,X))$ belongs to no shell.  Every other
point of $\mathcal E$ belongs to at least one shell.  In particular, if
$\alpha\ne\alpha_\ast$, then $(\alpha,h(\alpha))\in\mathcal E_n$ for some $n$.

\noindent
{\it Claim.} If $n_j\to\infty$ and
$
(\alpha_j,r_j)\in\mathcal E_{n_j},
$
then
$
r_j\to\htop(\sigma,X)
$
and
$
\alpha_j\to\alpha_\ast.
$

\begin{proof}[Proof of the claim]
The shell inequalities give
$$
0\le \htop(\sigma,X)-r_j\le\delta_{n_j-1}\longrightarrow0,
$$
so $r_j\to\htop(\sigma,X)$.  Since
$r_j\le h(\alpha_j)\le\htop(\sigma,X)$, it follows that
$h(\alpha_j)\to\htop(\sigma,X)$.  Let $\beta$ be a cluster point of
$(\alpha_j)$, which exists by compactness of $K$.  Upper semicontinuity gives
$$
\htop(\sigma,X)
=
\limsup_j h(\alpha_j)
\le
h(\beta)
\le
\htop(\sigma,X).
$$
Thus $h(\beta)=\htop(\sigma,X)$, and uniqueness of the maximizer yields
$\beta=\alpha_\ast$.  Every cluster point is therefore $\alpha_\ast$, so the
whole sequence converges to $\alpha_\ast$.
\end{proof}

Let $\Lambda_n$ be the layers from Proposition~\ref{mp:prop:separated-layers}.
If $(\alpha,r)\in\mathcal E_n$, then
$$
r\le \htop(\sigma,X)-\delta_n
<
\htop(\sigma,X)-\frac{\delta_n}{2}
<
\htop(\sigma,\Lambda_n).
$$
Let
$
\mathcal N=\{n\ge1:\mathcal E_n\ne\varnothing\}.
$
For every $n\in\mathcal N$, the function
$$
\rho_n:\mathcal E_n\to[0,\htop(\sigma,\Lambda_n)),
\qquad
\rho_n(\alpha,r)=r,
$$
is continuous.  Theorem~\ref{mp:thm:subcritical-entropy-sections}, applied to the transitive SFT $(\Lambda_n,\sigma|_{\Lambda_n})$, gives a continuous embedding
$
s_n:\mathcal E_n\longrightarrow \mathcal M_\sigma^e(\Lambda_n)
$
such that
$
h_{s_n(\alpha,r)}(\sigma)=r
$
for every 
$(\alpha,r)\in\mathcal E_n.
$

Define
$$
C
=
\{m_{\max}\}
\cup
\bigcup_{n\in\mathcal N}s_n(\mathcal E_n)
\subset \mathcal M_\sigma^e(X).
$$
Define $G:C\to K$ by
$$
G(m_{\max})=\alpha_\ast,
\qquad
G(s_n(\alpha,r))=\alpha.
$$
The map is well defined: each $s_n$ is injective, and the sets $\mathcal M_\sigma(\Lambda_n)$ are pairwise disjoint. We now prove that $C$ is compact and $G$ is continuous.

Let $(\mu_j)$ be a sequence in $C$.  If infinitely many terms lie in a fixed finite union
$
\{m_{\max}\}
\cup
\bigcup_{\substack{n\in\mathcal N\\ n\le N}}s_n(\mathcal E_n),
$
then compactness of that finite union gives a convergent subsequence.
It remains to consider a sequence of the form
$$
\mu_j=s_{n_j}(\alpha_j,r_j),
\qquad
n_j\to\infty,
\qquad
(\alpha_j,r_j)\in\mathcal E_{n_j}.
$$
Since $\mu_j\in \mathcal M_\sigma(\Lambda_{n_j})$ and
$
d_H(\mathcal M_\sigma(\Lambda_{n_j}),\{m_{\max}\})<2^{-n_j},
$
we have $\mu_j\to m_{\max}$.  Hence every sequence in $C$ has a convergent subsequence with limit in $C$, so $C$ is compact.

To prove continuity, let $\mu_j\to\mu$ in $C$.  Suppose first that $\mu\ne m_{\max}$.  Then $\mu$ belongs to a unique layer $s_n(\mathcal E_n)$, and, for all sufficiently large $j$, each $\mu_j$ also belongs to a layer.  These layer indices must be bounded; otherwise, along a subsequence they tend to infinity, and the Hausdorff estimate above would force that subsequence to converge to $m_{\max}$, contrary to $\mu_j\to\mu$.  Since finitely many pairwise disjoint compact layers are separated by a positive distance, $\mu_j$ eventually belongs to $s_n(\mathcal E_n)$.  The inverse of the embedding $s_n$ is continuous, and hence $G(\mu_j)\to G(\mu)$.

Now suppose that $\mu=m_{\max}$.  After omitting the indices $j$ for which
$\mu_j=m_{\max}$, write
$$
\mu_j=s_{n_j}(\alpha_j,r_j),
\qquad
(\alpha_j,r_j)\in\mathcal E_{n_j}.
$$
We must have $n_j\to\infty$.  Otherwise, after passing to a subsequence,
the indices $n_j$ would be constant.  The resulting sequence would lie in a
fixed compact layer and converge to $m_{\max}$, so that layer would contain
$m_{\max}$.  This is impossible, since every measure in the layer has entropy
strictly smaller than $\htop(\sigma,X)$, whereas
$h_{m_{\max}}(\sigma)=\htop(\sigma,X)$.  The preceding claim therefore gives
$$
G(\mu_j)=\alpha_j\longrightarrow\alpha_\ast=G(m_{\max}).
$$
Thus $G$ is continuous at $m_{\max}$ and hence on $C$.

For every $\mu\in C$, the pair $(C,G)$ satisfies the calibration inequality
$
h_\mu(\sigma)\le h(G(\mu)).
$
Indeed, equality holds at $m_{\max}$, since
$
h_{m_{\max}}(\sigma)=\htop(\sigma,X)=h(\alpha_\ast),
$
and if $\mu=s_n(\alpha,r)$, then
$
h_\mu(\sigma)=r\le h(\alpha)=h(G(\mu)).
$

It remains to verify that equality is attained for every parameter.  If
$\alpha=\alpha_\ast$, take $\mu_\alpha=m_{\max}$.  If
$\alpha\ne\alpha_\ast$, then $(\alpha,h(\alpha))\in\mathcal E_n$ for some
$n\in\mathcal N$; take
$
\mu_\alpha=s_n(\alpha,h(\alpha)).
$
In either case, $\mu_\alpha\in C$ is ergodic and satisfies
$
G(\mu_\alpha)=\alpha,
$
$
h_{\mu_\alpha}(\sigma)=h(\alpha).
$
\end{proof}

\subsection{Proof of Theorem~\ref{mp:thm:main-realization}}\label{mp:sec:calibrated-faces}

\begin{proof}[Proof of Theorem~\ref{mp:thm:main-realization}]
Fix $h\in C_{ms}^{m,H}$ and $N\in\mathbb N$. By
Proposition~\ref{mp:prop:compact-calibrated-pair}, there are a compact set
$C\subset\mathcal M_\sigma^e(X)$ and a continuous map $G:C\to K_h$ satisfying
the two calibration conditions in
Theorem~\ref{thm:calibrated-map-realization}.
By Theorem~\ref{thm:examples-H1-H5}, a transitive SFT satisfies
{\rm(H1)--(H5)}. Thus the second assertion of
Theorem~\ref{thm:calibrated-map-realization}, applied to the calibrated pair
above, gives
$
\Phi_1,\dots,\Phi_N\in C(X,\mathbb R^m)
$
with the required spectra, and these vector potentials are pairwise
non-cohomologous.
\end{proof}

\section{Whole-space pressure flexibility}\label{sec:app}

The entropy spectrum realization theorem has a direct convex dual
consequence.  This section first records the necessary duality and then proves
the multiparameter whole-space pressure realization theorem, which completely
resolves the problem of Kucherenko and Quas.

\subsection{Convex duality for entropy spectra}
Let $m\geq1$, and let
$F:\mathbb R^m\to\mathbb R$ be finite and convex.  Its
Legendre--Fenchel transform is
$$
F^*(\alpha)
=
\sup_{t\in\mathbb R^m}\{t\cdot\alpha-F(t)\}
\quad\text{for every }\alpha\in\mathbb R^m.
$$
We use the following standard form of the Fenchel--Moreau theorem.
\begin{lemma}\cite[Theorem 12.2]{Rockafellar1970}\label{lem:Fenchel-Moreau}
For every $t\in\mathbb R^m$, one has
$
F(t)
=
\sup_{\alpha\in\mathbb R^m}\{t\cdot\alpha-F^*(\alpha)\}.
$
Equivalently, $F=F^{**}$.
\end{lemma}

Let $F:\bbR^m\to\bbR$ be finite and convex.  A vector $v\in\bbR^m$ is a \emph{subgradient} of $F$ at $s\in\bbR^m$ if
$$
F(t)\ge F(s)+v\cdot(t-s)
\qquad\text{for every }t\in\bbR^m.
$$
The set of all such vectors is the \emph{subdifferential} $\partial F(s)$.  For $v\in\partial F(s)$, the affine function
$$
t\longmapsto F(s)+v\cdot(t-s)
=
\bigl(F(s)-v\cdot s\bigr)+v\cdot t
$$
is a supporting affine function of $F$, its graph is a supporting hyperplane to the graph of $F$, and its vertical-axis intercept is $F(s)-v\cdot s$.

\subsection{Differentiability at the origin}\label{sec:differentiability-at-origin}
Let $(X,f)$ be a dynamical system.
Suppose that the entropy map is upper semicontinuous and that the measure of
maximal entropy $m_{\max}$ is unique. Set $F(t):=P_f(t\cdot\Phi)$. For each
$t\in\mathbb R^m$, choose an arbitrary equilibrium state $\mu_t$ of
$t\cdot\Phi$. Such a state exists by upper semicontinuity of the entropy map.
No uniqueness of equilibrium state is assumed.  Put
$$
a=\int\Phi\,dm_{\max}
\quad\text{and}\quad
M=\max_{x\in X}\|\Phi(x)\|.
$$
The pressure variational principle gives
$$
\htop(f)+t\cdot a
\le F(t)
=h_{\mu_t}(f)+t\cdot\int\Phi\,d\mu_t
\le \htop(f)+M\|t\|.
$$
Consequently, $F(t)\to\htop(f)$ and
$h_{\mu_t}(f)\to\htop(f)$ as $t\to0$.  If $t_n\to0$, then every convergent
subsequence of $(\mu_{t_n})$ has a limit $\mu$ satisfying
$h_\mu(f)=\htop(f)$, by upper semicontinuity of the entropy map.  The
uniqueness of the measure of maximal entropy forces $\mu=m_{\max}$.
Therefore every such choice satisfies $\mu_t\to m_{\max}$ as $t\to0$.
Evaluating the pressure variational principle at $m_{\max}$ gives the first
inequality below, whereas the second follows from
$h_{\mu_t}(f)\le\htop(f)=F(0)$:
$$
0
\le F(t)-F(0)-t\cdot a
 =h_{\mu_t}(f)-\htop(f)
 +t\cdot\left(\int\Phi\,d\mu_t-a\right)
\le t\cdot\left(\int\Phi\,d\mu_t-a\right)
\le \|t\|\left\|\int\Phi\,d\mu_t-a\right\|.
$$
It follows directly that
$$
\lim_{t\to0}
\frac{\left|F(t)-F(0)-t\cdot a\right|}{\|t\|}=0.
$$
Thus $F$ is differentiable at the origin and
$\nabla F(0)=\int\Phi\,dm_{\max}$.  Hence differentiability at the origin is
necessary when the pressure is prescribed on all of $\mathbb R^m$.

\subsection{A dual reduction to entropy spectrum realization}

The following proposition isolates the convex dual argument common to the
multiparameter and scalar pressure realization results.

\begin{proposition}\label{prop:pressure-dual-reduction}
Let $m\ge1$ and $H>0$, and let $F:\mathbb R^m\to\mathbb R$ be finite and
convex.  Suppose that $F(0)=H$, that $F$ is differentiable at $0$, and that
the vertical intercept of every supporting hyperplane of $F$ is
nonnegative.  Define
$$
h(\alpha)
:=
\inf_{t\in\mathbb R^m}\bigl(F(t)-t\cdot\alpha\bigr)
=-F^*(\alpha),
\qquad
K:=\{\alpha\in\mathbb R^m:h(\alpha)>-\infty\}.
$$
Then $K$ is a nonempty compact convex set and
$
h|_K\in C_{ms}^{m,H},
$
with unique maximizer $\nabla F(0)$.  Moreover,
\begin{equation}\label{eq:dual-pressure-recovery}
F(t)
=
\sup_{\alpha\in K}\bigl(h(\alpha)+t\cdot\alpha\bigr)
\qquad\text{for every }t\in\mathbb R^m.
\end{equation}

If, in addition, $(X,f)$ has the saturated property for every invariant
measure, $\htop(f)=H$, and $\Phi\in C(X,\mathbb R^m)$ satisfies
$$
R(\Phi;X,f)=K
$$
and
$$
\htop(f,L(\Phi,\alpha;X,f))=h(\alpha)
\qquad\text{for every }\alpha\in K,
$$
then
$
P_f(t\cdot\Phi)=F(t)
$
for every 
$
t\in\mathbb R^m.
$
\end{proposition}

\begin{proof}
For each $t\in\mathbb R^m$, the function
$\alpha\mapsto F(t)-t\cdot\alpha$ is affine and continuous.  Hence $h$ is
concave and upper semicontinuous on $\mathbb R^m$, and $K$ is convex.

Set
$
C
=
\overline{\operatorname{conv}}
\left(\bigcup_{s\in\mathbb R^m}\partial F(s)\right).
$
If $v\in\partial F(s)$, write
$$
a(s,v):=F(s)-v\cdot s
$$
for the vertical intercept of the corresponding supporting hyperplane.
By hypothesis $a(s,v)\ge0$, while the supporting inequality at the origin
gives $a(s,v)\le F(0)=H$.  Let $e_1,\ldots,e_m$ be the standard basis
vectors.  The supporting inequalities at $e_i$ and $-e_i$ give
$$
-F(-e_i)\le v\cdot e_i\le F(e_i)
\qquad\text{for every }i=1,\ldots,m.
$$
Thus all subgradients of $F$ are uniformly bounded, and $C$ is a nonempty
compact convex set.  In particular, $F$ is Lipschitz.

\medskip
\noindent
{\it Claim.} $K=C$.

\begin{proof}[Proof of the claim]
Let $\alpha\in C$.  By the definition of $C$, there are convex
combinations
$$
\alpha_n=\sum_i\lambda_{n,i}v_{n,i},
\qquad
v_{n,i}\in\partial F(s_{n,i}),
$$
such that $\alpha_n\to\alpha$.  For every $t\in\mathbb R^m$, the supporting
inequalities and the nonnegativity of the vertical intercepts give
$$
F(t)
\ge
\sum_i\lambda_{n,i}
\bigl(a(s_{n,i},v_{n,i})+v_{n,i}\cdot t\bigr)
\ge
\alpha_n\cdot t.
$$
Letting $n\to\infty$ gives $F(t)\ge\alpha\cdot t$ for every $t$, and hence
$h(\alpha)\ge0$.  Thus $\alpha\in K$, proving $C\subset K$.

Conversely, suppose that $\alpha\notin C$.  Since $C$ is a nonempty compact
convex subset of $\mathbb R^m$, the strong separation theorem applied to
$C$ and $\{\alpha\}$ gives a nonzero vector $w\in\mathbb R^m$ such that
$$
\delta
:=
\alpha\cdot w-\sup_{v\in C}v\cdot w
>0.
$$
For $q>0$, choose $v_q\in\partial F(qw)$.  Since $v_q\in C$ and
$a(qw,v_q)\le H$, we have
$$
F(qw)-q\alpha\cdot w
=
a(qw,v_q)+q(v_q-\alpha)\cdot w
\le
H-q\delta.
$$
It follows that
$
h(\alpha)
\le
F(qw)-q\alpha\cdot w
\le
H-q\delta
$
for every $q>0$.  Letting $q\to\infty$ gives $h(\alpha)=-\infty$, so
$\alpha\notin K$.  This proves $K\subset C$ and hence $K=C$.
\end{proof}

Consequently, $K$ is nonempty, compact, and convex.

We next determine the range and the maximizers of $h$.  For every
$\alpha\in K=C$, the proof of the claim gives $h(\alpha)\ge0$, while
choosing $t=0$ in the definition of $h$ gives
$
h(\alpha)\le F(0)=H.
$
Moreover, because the expression defining $h(\alpha)$ equals $H$ at
$t=0$, we have
$$
\begin{aligned}
h(\alpha)=H
\quad\Longleftrightarrow\quad
F(t)-t\cdot\alpha\ge F(0)
\quad\text{for every }t\in\mathbb R^m
\quad\Longleftrightarrow\quad
\alpha\in\partial F(0).
\end{aligned}
$$
It remains to identify $\partial F(0)$.  For $0<s<1$, convexity gives
$
\frac{F(st)-F(0)}{s}\le F(t)-F(0).
$
Letting $s\downarrow0$ and using differentiability at the origin yields
the supporting inequality
$$
F(t)\ge F(0)+\nabla F(0)\cdot t
\qquad\text{for every }t\in\mathbb R^m,
$$
so $\nabla F(0)\in\partial F(0)$.  Conversely, if
$\alpha\in\partial F(0)$, then for every $v\in\mathbb R^m$ and $r>0$,
$
\frac{F(rv)-F(0)}{r}\ge\alpha\cdot v.
$
Letting $r\downarrow0$ yields
$\nabla F(0)\cdot v\ge\alpha\cdot v$.  Replacing $v$ by $-v$ gives the
reverse inequality, and hence $\alpha=\nabla F(0)$.  Therefore
$
\partial F(0)=\{\nabla F(0)\}
$
and
$$
h(\alpha)=H
\quad\Longleftrightarrow\quad
\alpha=\nabla F(0).
$$
In particular, $h(\nabla F(0))=H$, and this maximizer is unique.  Together
with the concavity and upper semicontinuity established above, this proves
that $h|_K\in C_{ms}^{m,H}$.

Since $h=-F^*$ on $K$ and $h=-\infty$ off $K$, the Fenchel--Moreau theorem
gives \eqref{eq:dual-pressure-recovery}.  Finally, suppose that $\Phi$ has
the stated dynamical properties.  Proposition~\ref{prop:shifts-facts} and
the pressure variational principle yield
$$
\begin{aligned}
P_f(t\cdot\Phi)
=
\sup_{\mu\in\mathcal M_f(X)}
\left(h_\mu(f)+t\cdot\int\Phi\,d\mu\right)
=
\sup_{\alpha\in K}\bigl(h(\alpha)+t\cdot\alpha\bigr)
=F(t),
\end{aligned}
$$
as required.
\end{proof}

\subsection{Proof of Theorem~\ref{mp:thm:whole-space-pressure}}
\begin{proof}[Proof of Theorem~\ref{mp:thm:whole-space-pressure}]
Fix $N\in\mathbb N$.  Proposition~\ref{prop:pressure-dual-reduction},
applied to $F$, gives a nonempty compact convex set
$K\subset\mathbb R^m$ and a function $h\in C_{ms}^{m,H}$.
By Theorem~\ref{mp:thm:main-realization}, there are pairwise
non-cohomologous vector potentials
$
\Phi_1,\ldots,\Phi_N\in C(X,\mathbb R^m)
$
such that, for every $j=1,\ldots,N$,
$$
R(\Phi_j;X,\sigma)=K
$$
and
$$
\htop(\sigma,L(\Phi_j,\alpha;X,\sigma))=h(\alpha)
\qquad\text{for every }\alpha\in K.
$$
By Theorem~\ref{thm:examples-H1-H5}, transitive SFTs have the saturated property for every
invariant measure.  Hence the final assertion of
Proposition~\ref{prop:pressure-dual-reduction}, applied to each $\Phi_j$,
gives
$$
P_\sigma(t\cdot\Phi_j)=F(t)
\qquad\text{for every }t\in\mathbb R^m
\text{ and }j=1,\ldots,N.
$$
Thus the $\Phi_j$ have all the required properties.
\end{proof}

\section{Scalar realization beyond SFTs}\label{sec:scalar-realization}
\subsection{Abstract scalar hypotheses}

The realization results in this section are established under structural
hypotheses shared by a broad class of standard symbolic and uniformly
hyperbolic systems; see Theorem~\ref{thm:examples-H1-H5} for concrete
examples.  This abstract framework separates the realization mechanism from
any particular model and highlights its three main ingredients:
thermodynamic formalism, saturated sets, and approximation by zero entropy
ergodic measures.  More precisely, we assume:
\begin{itemize}
    \item[(H1)] the entropy map $\mu\mapsto h_\mu(f)$ is upper
    semicontinuous on $\mathcal M_f(X)$;
    \item[(H2)] every Lipschitz potential on $X$ has a unique equilibrium
    state;
    \item[(H3)] $(X,f)$ has the saturated property for every invariant
    measure;
    \item[(H4)] there exist two distinct uniquely ergodic measures with zero
    entropy;
    \item[(H5)] ergodic measures with zero entropy are dense in
    $\mathcal M_f(X)$.
\end{itemize}
The abstract realization principle relies on hypotheses ${\rm(H1)}$ and
${\rm(H3)}$.  Hypotheses ${\rm(H2)}$ and ${\rm(H4)}$ are used to construct
the two-branch entropy path, while ${\rm(H5)}$ is needed only for the
multiplicity of
non-cohomologous realizations.

\subsection{Entropy monotone paths}

The next result supplies the thermodynamic ingredient for the realization
theorem. It constructs an entropy monotone path of ergodic measures joining
two uniquely ergodic zero entropy measures through the unique measure of
maximal entropy. This path will later be used as an entropy scale on which
the target spectrum is encoded.

\begin{theorem}\label{thm:periodic-measure-monotone-path}
	Let $(X,f)$ be a dynamical system
	satisfying {\rm(H1)--(H4)}.
	Denote by $m_{\max}$ the unique measure of maximal entropy of $(X,f)$.
	Then, for any two distinct uniquely ergodic zero entropy measures
	$\mu_L,\mu_R\in\mathcal M_f(X)$, there exists a continuous path
	$
	\{\nu_t\}_{t\in[-1,1]}\subset \mathcal M_f^e(X)
	$
	such that
	$$
	\nu_{-1}=\mu_L,\qquad \nu_0=m_{\max},\qquad \nu_1=\mu_R,
	$$
	the map
	$
	t\mapsto h_{\nu_t}(f)
	$
	is continuous on $[-1,1]$, nondecreasing on $[-1,0]$, and nonincreasing on $[0,1]$, and
	$
	\{\nu_t:t\in[-1,0]\}\cap \{\nu_t:t\in[0,1]\}=\{m_{\max}\}.
	$
\end{theorem}
\begin{proof}
	Write
		$
		A:=\operatorname{Supp}(\mu_L),
		~
		B:=\operatorname{Supp}(\mu_R).
		$
	Since $\mu_L$ and $\mu_R$ are distinct uniquely ergodic measures, $A$ and $B$ are disjoint closed sets. Define
	$$
	g(z):=\frac{d(z,B)-d(z,A)}{d(z,A)+d(z,B)}
	\qquad\text{for every }z\in X.
	$$
	Then $g$ is well defined and Lipschitz, with
	$$
	g|_A\equiv 1,\qquad g|_B\equiv -1,
	$$
	while $g(z)<1$ for $z\notin A$ and $g(z)>-1$ for $z\notin B$. It follows that
	$\mu_L$ is the unique maximizing measure of $g$, and $\mu_R$ is the unique minimizing measure of $g$.
	
	Now define
	$
	\varphi:=g-\int g\,dm_{\max}.
	$
	Then $\varphi$ is Lipschitz and
	$
	\int \varphi\,dm_{\max}=0.
	$
	Since subtracting a constant does not change maximizing or minimizing measures,
	$\mu_L$ is the unique maximizing measure of $\varphi$, and $\mu_R$ is the unique minimizing
	measure of $\varphi$.
	
	For each $s\in\mathbb R$, let $\mu_s$ denote the unique equilibrium state for $s\varphi$.
	By assumption, $\mu_s$ exists and is unique. In particular,
	$
	\mu_0=m_{\max},
	$
	and each $\mu_s$ is ergodic.
	
	We prove several properties of the family $\{\mu_s\}_{s\in\mathbb R}$.
	
	\medskip
	\noindent
	\medskip
	\noindent
	\medskip
	\noindent
	{\it Claim 1: the map $s\mapsto \mu_s$ is weak$^*$ continuous on $\mathbb R$.}
	\begin{proof}[Proof of Claim 1]
		For every $s\in\mathbb R$, the potential $s\varphi$ is Lipschitz and
		hence, by ${\rm(H2)}$, belongs to the set $\mathcal R_f$ in
		Proposition~\ref{prop:pressure}(4).  In the notation used there,
		$\mu_s=\nu_{s\varphi}$.  Since the map
		$$
		\mathbb R\longrightarrow C(X),
		\qquad s\longmapsto s\varphi,
		$$
		is continuous in the uniform norm, Proposition~\ref{prop:pressure}(4)
		implies that $s\mapsto\mu_s$ is weak$^*$ continuous.
	\end{proof}
	
	\medskip
	\noindent
	{\it Claim 2: the map $s\mapsto \int \varphi\,d\mu_s$ is nondecreasing on $\mathbb R$.}
	\begin{proof}[Proof of Claim 2]
		Let $s_1<s_2$, and write
		$$
		h_i:=h_{\mu_{s_i}}(f),\qquad
		a_i:=\int \varphi\,d\mu_{s_i}
		\qquad\text{for every }i\in\{1,2\}.
		$$
		Since $\mu_{s_i}$ is the equilibrium state for $s_i\varphi$, we have
		$$
		h_1+s_1a_1\ge h_2+s_1a_2,~
		h_2+s_2a_2\ge h_1+s_2a_1.
		$$
		Then we have
		$
		(s_2-s_1)(a_2-a_1)\ge 0.
		$
		Since $s_2-s_1>0$, it follows that
		$
		a_1\le a_2.
		$
	\end{proof}
	
	\medskip
	\noindent
	{\it Claim 3: the map $s\mapsto h_{\mu_s}(f)$ is continuous on $\mathbb R$,
		nondecreasing on $(-\infty,0]$, and nonincreasing on $[0,\infty)$.}
	\begin{proof}[Proof of Claim 3]
	With the notation above, if $0\le s_1<s_2$, then from
	$
	h_1+s_1a_1\ge h_2+s_1a_2
	$
	and $a_1\le a_2$ we obtain
	$
	h_1-h_2\ge s_1(a_2-a_1)\ge 0,
	$
	so
	$
	h_{\mu_{s_1}}(f)\ge h_{\mu_{s_2}}(f).
	$
	Thus $s\mapsto h_{\mu_s}(f)$ is nonincreasing on $[0,\infty)$.
	
	If $s_1<s_2\le 0$, then from
	$
	h_2+s_2a_2\ge h_1+s_2a_1
	$
	and $a_1\le a_2$ we get
	$
	h_2-h_1\ge s_2(a_1-a_2)\ge 0,
	$
	hence
	$
	h_{\mu_{s_1}}(f)\le h_{\mu_{s_2}}(f),
	$
	so $s\mapsto h_{\mu_s}(f)$ is nondecreasing on $(-\infty,0]$.
	
	Finally,
	$
	h_{\mu_s}(f)=P_f(s\varphi)-s\int \varphi\,d\mu_s.
	$
	By Proposition~\ref{prop:pressure}(2), the map
	$s\mapsto P_f(s\varphi)$ is continuous, and by Claim~1,
	the map
	$
	s\mapsto \int \varphi\,d\mu_s
	$
	is continuous. Therefore $s\mapsto h_{\mu_s}(f)$ is continuous on $\mathbb R$.
\end{proof}

	\medskip
	\noindent
	{\it Claim 4:}
	$
	\mu_s\to \mu_L\quad\text{as }s\to+\infty,
	$ $
	\mu_s\to \mu_R\quad\text{as }s\to-\infty.
	$
	\begin{proof}[Proof of Claim 4]
	Let
	$
	M:=\sup_{\mu\in\mathcal M_f(X)}\int \varphi\,d\mu.
	$
		Then
		$
		\int \varphi\,d\mu_L=M.
		$
		Suppose first that $s>0$. Since
		$
		h_{\mu_L}(f)=0,
		$
		the variational principle gives
	$$
	P_f(s\varphi)\ge h_{\mu_L}(f)+s\int \varphi\,d\mu_L=sM.
	$$
	On the other hand, for every $\mu\in\mathcal M_f(X)$, one has
	$
	h_\mu(f)\le \htop(f)$ and $\int \varphi\,d\mu\le M,
	$
	so
	$$
	P_f(s\varphi)\le \htop(f)+sM.
	$$
	Since
	$
	P_f(s\varphi)=h_{\mu_s}(f)+s\int \varphi\,d\mu_s,
	$
	we obtain
	$$
	sM\le h_{\mu_s}(f)+s\int \varphi\,d\mu_s\le \htop(f)+sM.
	$$
	As $0\le h_{\mu_s}(f)\le \htop(f)$, it follows that
	$
	0\le M-\int \varphi\,d\mu_s\le \frac{\htop(f)}{s}.
	$
	Hence
	$
	\int \varphi\,d\mu_s\to M$ as $s\to+\infty.
	$
	Now let $s_n\to+\infty$ and suppose $\mu_{s_n}\to \nu$ weak$^*$ along a subsequence.
	Then by continuity of $\varphi$,
	$
	\int \varphi\,d\nu=\lim_{n\to\infty}\int \varphi\,d\mu_{s_n}=M.
	$
	Since $\mu_L$ is the unique maximizing measure of $\varphi$, we must have $\nu=\mu_L$.
	Therefore
	$
	\mu_s\to \mu_L$ as $s\to+\infty.
	$
	Similarly, we have
	$
	\mu_s\to \mu_R$ as $s\to-\infty.
	$
\end{proof}

	\medskip
	\noindent
	{\it Claim 5:}
	$
	h_{\mu_s}(f)\to 0\quad\text{as }s\to\pm\infty.
	$
	\begin{proof}[Proof of Claim 5]
	By Claim~4, $\mu_s\to\mu_L$ as $s\to+\infty$. Since the entropy map is upper semicontinuous,
	$
	\limsup\limits_{s\to+\infty} h_{\mu_s}(f)\le h_{\mu_L}(f)=0.
	$
	As entropy is always nonnegative, we obtain
	$
	h_{\mu_s}(f)\to 0$ as $s\to+\infty.
	$
	Similarly, we have
	$
	h_{\mu_s}(f)\to 0$ as $s\to-\infty.
	$
\end{proof}

	\medskip
	\noindent
	{\it Claim 6: if $s_-<0<s_+$ and $\mu_{s_-}=\mu_{s_+}$, then}
	$
	\mu_{s_-}=\mu_{s_+}=m_{\max}.
	$
	\begin{proof}[Proof of Claim 6]
	Since $\int \varphi\,dm_{\max}=0$, we have
	$
	P_f(s\varphi)\ge h_{m_{\max}}(f)+s\int \varphi\,dm_{\max}=\htop(f)
	$
	for every $s\in\mathbb R$. On the other hand,
	$$
	P_f(s\varphi)=h_{\mu_s}(f)+s\int \varphi\,d\mu_s,
	\qquad
	h_{\mu_s}(f)\le \htop(f),
	$$
	so
	$
	s\int \varphi\,d\mu_s=P_f(s\varphi)-h_{\mu_s}(f)\ge 0.
	$
	Hence
	$$
	\int \varphi\,d\mu_s\ge 0\quad\text{if }s>0,\qquad
	\int \varphi\,d\mu_s\le 0\quad\text{if }s<0.
	$$
	
	Now suppose $s_-<0<s_+$ and $\mu_{s_-}=\mu_{s_+}=:\nu$. Then
	$
	\int \varphi\,d\nu=0.
	$
	Therefore
	$
	P_f(s_+\varphi)=h_\nu(f)+s_+\int \varphi\,d\nu=h_\nu(f).
	$
	On the other hand,
	$
	P_f(s_+\varphi)\ge \htop(f).
	$
	So we obtain
	$
	h_\nu(f)=\htop(f).
	$
	By uniqueness of the measure of maximal entropy, $\nu=m_{\max}$.
\end{proof}

	We now define
	$$
	\nu_t=
	\begin{cases}
		\mu_{-\tan(\pi t/2)}, & t\in(-1,1),\\
		\mu_L, & t=-1,\\
		\mu_R, & t=1.
	\end{cases}
	$$
	Since
	$$
	-\tan(\pi t/2)\to +\infty\quad\text{as }t\to -1^+,
	\qquad
	-\tan(\pi t/2)\to -\infty\quad\text{as }t\to 1^-,
	$$
	Claims~1 and 4 imply that $t\mapsto \nu_t$ is continuous on $[-1,1]$. Also,
	$$
	\nu_{-1}=\mu_L,\qquad \nu_0=\mu_0=m_{\max},\qquad \nu_1=\mu_R.
	$$
	Since each $\mu_s$ is a unique equilibrium state, it is ergodic, and hence
	$
		\nu_t\in\mathcal M_f^e(X)$ for any $t\in[-1,1]$.

	Now $t\mapsto -\tan(\pi t/2)$ is strictly decreasing on $(-1,1)$.
	By Claim~3, the map
	$
	s\longmapsto h_{\mu_s}(f)
	$
	is continuous on $\mathbb R$, nondecreasing on $(-\infty,0]$, and nonincreasing on
	$[0,\infty)$. Therefore
	$
	t\longmapsto h_{\nu_t}(f)
	$
	is nondecreasing on $[-1,0]$ and nonincreasing on $[0,1]$.
	
	It remains to prove continuity of $t\mapsto h_{\nu_t}(f)$ at the endpoints
	$t=\pm1$. For $t\in(-1,1)$ we have
	$
	h_{\nu_t}(f)=h_{\mu_{-\tan(\pi t/2)}}(f).
	$
	As $t\to -1^+$, we have $-\tan(\pi t/2)\to +\infty$, so by Claim~5,
	$$
	h_{\nu_t}(f)=h_{\mu_{-\tan(\pi t/2)}}(f)\to 0
	= h_{\mu_L}(f)=h_{\nu_{-1}}(f).
	$$
	Similarly, as $t\to 1^-$, we have $-\tan(\pi t/2)\to -\infty$, and again by Claim~5,
	$$
	h_{\nu_t}(f)=h_{\mu_{-\tan(\pi t/2)}}(f)\to 0
	= h_{\mu_R}(f)=h_{\nu_1}(f).
	$$
	Hence
	$
	t\longmapsto h_{\nu_t}(f)
	$
	is continuous on $[-1,1]$.
	
	Finally,
	$$
	\{\nu_t:t\in[-1,0]\}=\{\mu_s:s\in[0,\infty]\},
	\qquad
	\{\nu_t:t\in[0,1]\}=\{\mu_s:s\in[-\infty,0]\},
	$$
	where we use the conventions $\mu_{+\infty}:=\mu_L$ and
	$\mu_{-\infty}:=\mu_R$.  Since $\mu_L$ and $\mu_R$ are distinct
	zero entropy measures and $m_{\max}$ is the unique measure of maximal
	entropy, we have $\htop(f)>0$ and hence
	$\mu_L,\mu_R\ne m_{\max}$.  The uniqueness of the maximizing and
	minimizing measures of $\varphi$ therefore gives
	$$
	\int\varphi\,d\mu_L>\int\varphi\,dm_{\max}=0
	>\int\varphi\,d\mu_R.
	$$
	Together with the sign inequalities proved in Claim~6, this shows that
	$$
	\mu_L\notin\{\mu_s:s\in[-\infty,0]\},
	\qquad
	\mu_R\notin\{\mu_s:s\in[0,+\infty]\}.
	$$
	Now let a measure belong to both branches, and write it as
	$\mu_{s_+}=\mu_{s_-}$ with
	$s_+\in[0,+\infty]$ and $s_-\in[-\infty,0]$.  The preceding endpoint
	exclusions show that both parameters are finite.  If either parameter is
	zero, the common measure is $\mu_0=m_{\max}$; otherwise
	$s_-<0<s_+$, and Claim~6 again gives that the common measure is
	$m_{\max}$.  Conversely, $m_{\max}=\mu_0$ belongs to both branches.
	Consequently,
	$$
	\{\nu_t:t\in[-1,0]\}\cap \{\nu_t:t\in[0,1]\}=\{m_{\max}\}.
	$$
	This completes the proof.
\end{proof}

\begin{remark}\label{prop:entropy-continuous-on-path-image}
	Let
	$
	C:=\{\nu_t:t\in[-1,1]\}\subset \mathcal M_f(X),
	$
	where $\{\nu_t\}_{t\in[-1,1]}$ is the path given by
	Theorem~\ref{thm:periodic-measure-monotone-path}. Then the entropy map
	$
	\mu\longmapsto h_\mu(f)
	$
	is continuous on $C$ with respect to the weak$^*$ topology.
\end{remark}

\subsection{Proof of Theorem~\ref{thm:main}}

\begin{proof}[Proof of Theorem~\ref{thm:main}]
Set $H=\htop(f)$ and write $K_h=[a,b]$.  Let $\alpha_\ast$ be the
unique maximizer of $h$.  The defining properties of $C_{ms}^{1,H}$ imply
that $h$ is continuous on $[a,b]$; concavity and uniqueness of the maximizer
then imply that
$$
h_-:=h|_{[a,\alpha_\ast]}
\quad\text{and}\quad
h_+:=h|_{[\alpha_\ast,b]}
$$
are, respectively, strictly increasing and strictly decreasing.  Thus
$$
h_-:[a,\alpha_\ast]\longrightarrow[h(a),H],
\qquad
h_+:[\alpha_\ast,b]\longrightarrow[h(b),H]
$$
are homeomorphisms.

Choose distinct zero entropy measures $\mu_L$ and $\mu_R$ as in
{\rm(H4)}.  Theorem~\ref{thm:periodic-measure-monotone-path} provides a
continuous path
$
\{\nu_t\}_{t\in[-1,1]}\subset\mathcal M_f^e(X)
$
from $\mu_L$ to $\mu_R$ through the unique measure of maximal entropy
$m_{\max}=\nu_0$.  Set
$$
e(t):=h_{\nu_t}(f).
$$
Then $e$ is continuous, nondecreasing on $[-1,0]$, and nonincreasing on
$[0,1]$, with
$$
e(-1)=e(1)=0,
\qquad
e(0)=H.
$$
Moreover, the two branches of the path meet only at $m_{\max}$.

Choose $t_-\in[-1,0]$ and $t_+\in[0,1]$ such that
$$
e(t_-)=h(a),
\qquad
e(t_+)=h(b),
$$
and set
$$
C_-:=\{\nu_t:t\in[t_-,0]\},
\qquad
C_+:=\{\nu_t:t\in[0,t_+]\},
\qquad
C:=C_-\cup C_+.
$$
These sets are compact, $C\subset\mathcal M_f^e(X)$, and
$
C_-\cap C_+=\{m_{\max}\}.
$
The monotonicity and continuity of $e$ give
$$
\{h_\mu(f):\mu\in C_-\}=[h(a),H],
\qquad
\{h_\mu(f):\mu\in C_+\}=[h(b),H].
$$

Define $G:C\to K_h$ by
$$
G(\mu)=
\begin{cases}
h_-^{-1}(h_\mu(f)),&\mu\in C_-,\\
h_+^{-1}(h_\mu(f)),&\mu\in C_+.
\end{cases}
$$
This is well defined because the two formulas agree at $m_{\max}$, where
both take the value $\alpha_\ast$.  By
Remark~\ref{prop:entropy-continuous-on-path-image}, the entropy map is
continuous on $C$.  Hence each restriction of $G$ is continuous, and the
pasting lemma shows that $G$ is continuous on $C$.  By construction,
$$
h_\mu(f)=h(G(\mu))
\qquad\text{for every }\mu\in C,
$$
and
$$
G(C_-)= [a,\alpha_\ast],
\qquad
G(C_+)= [\alpha_\ast,b].
$$
In particular, $G(C)=K_h$, so $(C,G)$ satisfies both calibration conditions
in Theorem~\ref{thm:calibrated-map-realization}.

Hypotheses {\rm(H1)} and {\rm(H3)} allow us to apply that theorem with
$m=1$.  We obtain $\varphi\in C(X)$ such that
$$
R(\varphi;X,f)=K_h,
\qquad
\htop(f,L(\varphi,\alpha;X,f))=h(\alpha)
\quad\text{for every }\alpha\in K_h.
$$
If {\rm(H5)} also holds, the final assertion of the same theorem yields,
for every $N\in\mathbb N$, $N$ pairwise non-cohomologous potentials with
the same rotation interval and entropy spectrum.
\end{proof}

\subsection{Proof of Corollary~\ref{cor:KQ-strengthened}}

\begin{proof}[Proof of Corollary~\ref{cor:KQ-strengthened}]
Apply Proposition~\ref{prop:pressure-dual-reduction} with $m=1$.  It gives
a nonempty compact interval $I\subset\mathbb R$ and a function
$h\in C_{ms}^{1,H}$.
By Theorem~\ref{thm:main}, there is $\varphi\in C(X)$ such that
$$
R(\varphi;X,f)=I
$$
and
$$
\htop(f,L(\varphi,\alpha;X,f))=h(\alpha)
\qquad\text{for every }\alpha\in I.
$$
Hypothesis {\rm(H3)} is precisely the saturated property required for the
pressure conclusion of Proposition~\ref{prop:pressure-dual-reduction}.
Consequently,
$$
P_f(t\varphi)=F(t)
\qquad\text{for every }t\in\mathbb R.
$$

If {\rm(H5)} also holds, fix $N\in\mathbb N$.  The multiplicity conclusion
of Theorem~\ref{thm:main} provides pairwise non-cohomologous potentials
$\varphi_1,\ldots,\varphi_N$ realizing the same interval $I$ and spectrum
$h$.  Applying Proposition~\ref{prop:pressure-dual-reduction} to each of
them gives
$$
P_f(t\varphi_j)=F(t)
\qquad\text{for every }t\in\mathbb R
\text{ and }j=1,\ldots,N.
$$
This completes the proof.
\end{proof}

\subsection{Examples}
\begin{theorem}\label{thm:examples-H1-H5}
	The following classes of systems satisfy {\rm(H1)--(H5)}:
	\begin{enumerate}
		\item[(E1)] full shifts;
		\item[(E2)] transitive SFTs;
		\item[(E3)] transitive sofic shifts;
		\item[(E4)] $\beta$-shifts;
		\item[(E5)] $S$-gap shifts;
		\item[(E6)] transitive Anosov diffeomorphisms;
		\item[(E7)] transitive expanding maps.
	\end{enumerate}
\end{theorem}
\begin{proof}
	We recall the standard results that imply the five hypotheses. The upper semicontinuity of the entropy map in {\rm(H1)} is standard for expansive systems. See, for instance, \cite{Bowen2008,DenkerGrillenbergerSigmund1976,Walters1982,LindMarcus2021}.
	
	For {\rm(H2)}, Lipschitz potentials on the symbolic systems, with respect to the standard symbolic metrics, have the Bowen property or summable variation. Hence the standard thermodynamic formalism for potentials with the Bowen property, together with its non-uniform extensions, gives uniqueness of equilibrium states for such potentials; see \cite{Bowen2008,Walters1978,ClimenhagaThompson2013,ClimenhagaThompsonYamamoto2017}. The corresponding uniqueness theorem for Lipschitz potentials on transitive Anosov diffeomorphisms and transitive expanding maps is classical; see \cite{Bowen2008,Bowen1974}.
	
	For each system listed above, the saturated property in {\rm(H3)} follows from the appropriate combination of the specification property, the almost product property, the shadowing property, and suitable non-uniform structures, together with the corresponding saturated set results; see \cite{PfisterSullivan2007,ZhaoChen2018,DongHouTian2022}.

	Finally, {\rm(H4)} and {\rm(H5)} follow from approximation by periodic measures. In the systems listed above, periodic measures are ergodic, have zero entropy, and are uniquely ergodic on their supports.  There are at least two distinct periodic orbits in each nontrivial case, giving {\rm(H4)}.  Moreover, periodic measures, and hence zero entropy ergodic measures, are dense in the set of invariant probability measures in these systems; see \cite{Sigmund1970,GelfertKwietniak2018,HouTianYuan2023}. 
\end{proof}

\section{Stability properties of the spectrum map}\label{Section:discontinuity}
\label{sec:maximizing-discontinuity}
This section studies the stability of the spectrum map and proves
Theorems~\ref{thm:lower-continuous} and~\ref{thm:not-continuous}. We first
establish lower semicontinuity for vector-valued potentials of every finite
dimension when the entropy map is upper semicontinuous and the system has the
saturated property. We then prove that upper semicontinuity fails on a dense
subset of scalar potentials for nontrivial transitive two-sided SFTs.

We use the notation $\Gamma_h$, $\widehat\Gamma_h$, $E_m$, and
$\mathscr E_m$ introduced in
Subsection~\ref{subsec:intro-spectrum-stability}.  By
Proposition~\ref{prop:why-ms}, $E_m(\Phi)$ is upper semicontinuous and
concave on the compact convex set $R(\Phi;X,f)$.  Hence
$\widehat\Gamma_{E_m(\Phi)}$ is compact and the map $\mathscr E_m$ is well
defined.  When $m=1$, we abbreviate $E_1$ to $E$; in this case the spectrum
is continuous and $\widehat\Gamma_{E(\varphi)}=\Gamma_{E(\varphi)}$.
\subsection{Proof of Theorem~\ref{thm:lower-continuous}}
\begin{proof}[Proof of Theorem~\ref{thm:lower-continuous}]
Fix $m\geq1$ and $\Phi\in C(X,\mathbb R^m)$, and let
$\Phi_n\to\Phi$ uniformly.  We first show that every point of
$\Gamma_{E_m(\Phi)}$ can be approximated by points of
$\Gamma_{E_m(\Phi_n)}$.

Fix
$
p=(\alpha,E_m(\Phi)(\alpha))\in\Gamma_{E_m(\Phi)}.
$
By Proposition~\ref{prop:shifts-facts} and upper semicontinuity of the
entropy map, there exists $\mu_\ast\in\mathcal M_f(X)$ such that
$$
\int\Phi\,d\mu_\ast=\alpha,
\qquad
h_{\mu_\ast}(f)=E_m(\Phi)(\alpha).
$$
For each $n$, set
$
\alpha_n:=\int\Phi_n\,d\mu_\ast.
$
Then $\alpha_n\to\alpha$.  Choose a measure $\nu_n$ so that
$$
\int\Phi_n\,d\nu_n=\alpha_n,
\qquad
h_{\nu_n}(f)=E_m(\Phi_n)(\alpha_n).
$$
Then we have
$$
E_m(\Phi_n)(\alpha_n)
\geq h_{\mu_\ast}(f)
=E_m(\Phi)(\alpha).
$$

To obtain the reverse asymptotic inequality, take a subsequence along which
$E_m(\Phi_n)(\alpha_n)$ converges to its limsup.  After passing to a
further subsequence, compactness of $\mathcal M_f(X)$ gives
$\nu_n\to\nu\in\mathcal M_f(X)$.  Uniform convergence of the potentials and
weak$^*$ convergence of the measures imply
$
\int\Phi\,d\nu
=\lim_{n\to\infty}\int\Phi_n\,d\nu_n
=\lim_{n\to\infty}\alpha_n
=\alpha.
$
Consequently, the conditional variational principle and upper
semicontinuity of entropy give
$$
\limsup_{n\to\infty}E_m(\Phi_n)(\alpha_n)
=\limsup_{n\to\infty}h_{\nu_n}(f)
\leq h_\nu(f)
\leq E_m(\Phi)(\alpha).
$$
It follows that
$
E_m(\Phi_n)(\alpha_n)\longrightarrow E_m(\Phi)(\alpha).
$
Thus the points
$
p_n:=(\alpha_n,E_m(\Phi_n)(\alpha_n))
\in\Gamma_{E_m(\Phi_n)}
$
converge to $p$.
We have therefore proved that
$$
\operatorname{dist}
\bigl(p,\widehat\Gamma_{E_m(\Phi_n)}\bigr)\longrightarrow0
\qquad\text{for every }p\in\Gamma_{E_m(\Phi)}.
$$
Fix $\varepsilon>0$.  Since $\Gamma_{E_m(\Phi)}$ is dense in the compact set
$\widehat\Gamma_{E_m(\Phi)}$, there are
$q_1,\ldots,q_k\in\Gamma_{E_m(\Phi)}$ such that every
$p\in\widehat\Gamma_{E_m(\Phi)}$ satisfies $|p-q_i|<\varepsilon/2$ for some
$i$.  For all sufficiently large $n$,
$$
\operatorname{dist}
\bigl(q_i,\widehat\Gamma_{E_m(\Phi_n)}\bigr)<\varepsilon/2
\qquad\text{for every }i=1,\ldots,k.
$$
For each $p\in\widehat\Gamma_{E_m(\Phi)}$, choose $i$ as above.  Then
$$
\operatorname{dist}
\bigl(p,\widehat\Gamma_{E_m(\Phi_n)}\bigr)
\le |p-q_i|
+\operatorname{dist}
\bigl(q_i,\widehat\Gamma_{E_m(\Phi_n)}\bigr)
<\varepsilon.
$$
Thus the convergence is uniform on $\widehat\Gamma_{E_m(\Phi)}$.  Hence
$
e\bigl(\widehat\Gamma_{E_m(\Phi)},
\widehat\Gamma_{E_m(\Phi_n)}\bigr)\longrightarrow0.
$
Thus $\mathscr E_m$ is lower semicontinuous at $\Phi$.
\end{proof}

\subsection{Perturbations creating many maximizing measures}
The following lemma is a direct consequence of the perturbation argument in \cite{Shinoda2018}.
Indeed, the continuous path $\{\mu_t\}_{t\in[0,1]} \subset \mathcal M_\sigma^e(X)$ with countable fibers gives rise to a non-atomic Borel probability measure on $\mathcal M_\sigma(X)$ that is supported on ergodic measures and has $\mu$ in its support.
Therefore one can apply the same Bishop--Phelps argument as in \cite[Proposition 4]{Shinoda2018} to obtain, after an arbitrarily small perturbation of the potential, uncountably many ergodic maximizing measures along this path.
\begin{lemma}\label{lem:Shinoda2018}
	Let $(X,\sigma)$ be a subshift, let $\varphi\in C(X)$, and let
	$
	\mu\in \mathcal M_{\max}(\varphi)\cap \mathcal M_\sigma^e(X).
	$
	Assume that there exists a continuous map
	$
	[0,1]\ni t\mapsto \mu_t\in \mathcal M_\sigma^e(X)
	$
	such that $\mu_0=\mu$ and, for every
	$$
	\nu\in \{\mu_t:t\in[0,1]\},
	$$
	the set
	$
	\{t\in[0,1]:\mu_t=\nu\}
	$
	is countable. Then for any $\varepsilon>0$, there exists $\psi\in C(X)$ such that
	$
	\|\varphi-\psi\|_\infty<\varepsilon
	$
	and $\mathcal M_{\max}(\psi)$ contains uncountably many elements of
	$
	\{\mu_t\}_{t\in[0,1]}.
	$
\end{lemma}
Combining the monotone path construction with the perturbation argument, we
obtain the following density statement.
\begin{theorem}\label{thm:beta-shift-uncountably-many-distinct-entropies}
	Let $(X,\sigma)$ be a nontrivial transitive two-sided SFT. Then there exists a dense subset
	$D\subset C(X)$ such that for every $\varphi\in D$, the set
	$\mathcal M_{\max}(\varphi)$ contains uncountably many ergodic measures with full support whose
	metric entropies are pairwise distinct.
\end{theorem}
\begin{proof}
	Set
	$
	\mathcal P
	:=
	\left\{
	\mu_x :
	x\in X \text{ is a periodic point}\right\}.
	$
	Let $p$ be the period of $(X,\sigma)$, and let
	$$
	X=X_0\sqcup X_1\sqcup\cdots\sqcup X_{p-1}
	$$
	be its cyclic decomposition. Put $T=\sigma^p|_{X_0}$. Then
	$(X_0,T)$ is a mixing two-sided SFT. By \cite{Sigmund1970}, the periodic
	orbit measures are dense in $\mathcal M_T(X_0)$. Moreover,
	Lemma~\ref{mp:lem:cyclic-measure-transfer} shows that the cyclic averaging
	map
	$$
	\mathcal A:\mathcal M_T(X_0)\longrightarrow\mathcal M_\sigma(X),
	\qquad
	\mathcal A(\rho)=\frac1p\sum_{i=0}^{p-1}(\sigma^i)_*\rho,
	$$
	is an affine homeomorphism. If $\rho$ is a periodic orbit measure for
	$T$, then $\mathcal A(\rho)$ is a periodic orbit measure for $\sigma$.
	Consequently, $\mathcal P$ is dense in $\mathcal M_\sigma(X)$. Since
	$\mathcal P\subset \mathcal M_\sigma^e(X)$, it follows that
	$\mathcal P$ is dense in $\mathcal M_\sigma^e(X)$ as well. Then
	by \cite[Lemma 3.3]{Morris2010}, the set
	$$
	\mathcal D_{\mathrm{per}}
	:=
	\left\{
	\phi\in C(X):
	\mathcal M_{\max}(\phi)=\{\mu_x\}
	\text{ for some }\mu_x\in \mathcal P
	\right\}
	$$
	is dense in $C(X)$.
	
	Fix $\phi_0\in C(X)$ and $\varepsilon>0$. Choose
	$\phi_1\in \mathcal D_{\mathrm{per}}$ such that
	$
	\|\phi_1-\phi_0\|_\infty<\frac{\varepsilon}{2}.
	$
	Then there exists a
	periodic point $x\in X$ such that
	$
	\mathcal M_{\max}(\phi_1)=\{\mu_x\}.
	$
	Since $\mu_x$ is supported on a periodic orbit, we have
	$
	\mu_x\in \mathcal M_{\max}(\phi_1)\cap \mathcal M_\sigma^e(X).
	$
	Moreover, by Lemma~\ref{lem:periodic-measure-monotone-path-transitive-sft}, there exists
	a continuous path
	$$
	[0,1]\ni t\mapsto \mu_t\in \mathcal M_\sigma^e(X)
	$$
	such that $\mu_0=\mu_x$, the map
	$
	t\mapsto h_{\mu_t}(\sigma)
	$
	is strictly increasing on $[0,1]$, and $\operatorname{Supp}(\mu_t)=X$ for any $t\in(0,1]$. In particular, the path is injective, and hence
	for every
	$
	\nu\in \{\mu_t:t\in[0,1]\}
	$
	the set
	$
	\{t\in[0,1]:\mu_t=\nu\}
	$
	is countable. Therefore all the assumptions of Lemma~\ref{lem:Shinoda2018} are
	satisfied. Applying Lemma~\ref{lem:Shinoda2018} to $\phi_1$ and $\mu_x$, we obtain
	$\phi\in C(X)$ such that
	$
	\|\phi-\phi_1\|_\infty<\frac{\varepsilon}{2},
	$
	and the set
	$$
	\mathcal U
	:=
	\mathcal M_{\max}(\phi)\cap\{\mu_t:t\in[0,1]\}
	$$
	is uncountable.  Removing the single possible endpoint measure
	$\mu_0=\mu_x$ leaves the uncountable set
	$$
	\mathcal U\setminus\{\mu_0\}
	\subset
	\mathcal M_{\max}(\phi)\cap\{\mu_t:t\in(0,1]\}.
	$$
	Every measure in this set has full support, while the strict increase of
	$t\mapsto h_{\mu_t}(\sigma)$ shows that their metric entropies are pairwise
	distinct.
	Finally,
	$
	\|\phi-\phi_0\|_\infty
	\le
	\|\phi-\phi_1\|_\infty+\|\phi_1-\phi_0\|_\infty
	<
	\varepsilon.
	$
\end{proof}

\subsection{Proof of Theorem~\ref{thm:not-continuous}}

\begin{proof}[Proof of Theorem~\ref{thm:not-continuous}]
Now let $D\subset C(X)$ be the dense subset given by
Theorem~\ref{thm:beta-shift-uncountably-many-distinct-entropies}, and fix
$\varphi\in D$. Then $\mathcal M_{\max}(\varphi)$ contains uncountably many
ergodic measures with pairwise distinct entropies, so we may choose
$$
\mu_1,\mu_2\in \mathcal M_{\max}(\varphi)\cap \mathcal M_\sigma^e(X)
$$
such that
$
h_{\mu_1}(\sigma)<h_{\mu_2}(\sigma).
$
Let
$
\alpha_\ast:=\max_{\nu\in\mathcal M_\sigma(X)}\int\varphi\,d\nu.
$
Then
$
\int\varphi\,d\mu_1=\int\varphi\,d\mu_2=\alpha_\ast,
$
hence by Proposition~\ref{prop:shifts-facts},
$$
E(\varphi)(\alpha_\ast)\ge h_{\mu_2}(\sigma)>h_{\mu_1}(\sigma).
$$

By \cite[Theorem 3.7]{Jenkinson2006}, there exists $\psi\in C(X)$ such that
$
\mathcal M_{\max}(\psi)=\{\mu_1\}.
$
Replacing $\psi$ by
$
\psi-\max_{\nu\in\mathcal M_\sigma(X)}\int \psi\,d\nu+\alpha_\ast,
$
we may assume
$$
\max_{\nu\in\mathcal M_\sigma(X)}\int \psi\,d\nu=\alpha_\ast
\quad\text{and}\quad
\mathcal M_{\max}(\psi)=\{\mu_1\}.
$$

For $t\in(0,1)$, define
$$
\varphi_t:=(1-t)\varphi+t \psi.
$$
Then $\varphi_t\to\varphi$ in $C(X)$ as $t\to0$. Moreover,
$\mu_1\in\mathcal M_{\max}(\varphi_t)$, since
$
\int\varphi_t\,d\mu_1=(1-t)\alpha_\ast+t\alpha_\ast=\alpha_\ast.
$
If $\nu\in\mathcal M_{\max}(\varphi_t)$, then
$$
\int\varphi_t\,d\nu=(1-t)\int\varphi\,d\nu+t\int \psi\,d\nu\le (1-t)\alpha_\ast+t\alpha_\ast=\alpha_\ast.
$$
Since equality holds, we have
$
\int\varphi\,d\nu=\alpha_\ast
$
and
$
\int \psi\,d\nu=\alpha_\ast,
$
so $\nu\in\mathcal M_{\max}(\psi)=\{\mu_1\}$. Therefore
$
\mathcal M_{\max}(\varphi_t)=\{\mu_1\}.
$
In particular, by Proposition~\ref{prop:shifts-facts}, for every $t\in(0,1)$,
$
E(\varphi_t)(\alpha_\ast)=h_{\mu_1}(\sigma),
$
so the point
$$
p:=(\alpha_\ast,h_{\mu_1}(\sigma))
$$
belongs to $\Gamma_{E(\varphi_t)}$ for every $t\in(0,1)$.

On the other hand,
$
E(\varphi)(\alpha_\ast)\ge h_{\mu_2}(\sigma)>h_{\mu_1}(\sigma),
$
hence
$
p\notin \Gamma_{E(\varphi)}.
$
Since $\Gamma_{E(\varphi)}$ is compact in $\mathbb R^2$, it follows that
$$
\delta:=\operatorname{dist}(p,\Gamma_{E(\varphi)})>0.
$$
Consequently,
$$
e\bigl(\Gamma_{E(\varphi_t)},\Gamma_{E(\varphi)}\bigr)
\ge
\operatorname{dist}(p,\Gamma_{E(\varphi)})
=
\delta
\qquad\text{for every }t\in(0,1).
$$
Since $\varphi_t\to\varphi$ as $t\to0$, the map $\mathscr E_1$ is not upper
semicontinuous at $\varphi$.  
\end{proof}

\section*{Appendix: Proof of
Lemma~\ref{mp:lem:one-measure-sft-approximation}}
\label{app:proof-local-sft-approximation}
\addcontentsline{toc}{section}{Appendix: Proof of
Lemma~\ref{mp:lem:one-measure-sft-approximation}}

\begin{proof}[Proof of Lemma~\ref{mp:lem:one-measure-sft-approximation}]
Let $A$ be an irreducible zero-one transition matrix on the finite alphabet
$\mathcal A$ defining $X=X_A$.
Choose an integer $s\ge1$, functions $\varphi_1,\ldots,\varphi_s\in C(X)$, and $\eps>0$ such that
$$
\mathcal U_0
=
\left\{
\nu\in \mathcal M_\sigma(X):
\left|\int\varphi_i\,d\nu-\int\varphi_i\,d\mu\right|<\eps
\text{ for }1\le i\le s
\right\}
\subset U.
$$
Locally constant functions are uniformly dense in $C(X)$.  Hence, after choosing a common radius $r\ge0$, there are functions $\psi_1,\ldots,\psi_s$ depending only on the coordinates in $[-r,r]$ such that
$$
\|\varphi_i-\psi_i\|_\infty<\frac{\eps}{8}
\qquad\text{for every }1\le i\le s.
$$
Put
$$
a_i=\int\psi_i\,d\mu,
\qquad
C=1+\max_{1\le i\le s}\|\psi_i\|_\infty.
$$
In particular, $|a_i|\le C$.  Let
$\mathcal A^*=\bigcup_{n\geq0}\mathcal A^n$, where
$\mathcal A^0=\{\emptyset\}$.  For $u,v\in\mathcal A$ and
$c\in\mathcal A^*$, the expression $ucv$ denotes their concatenation.
Since $A$ is irreducible, the number
$$
D
=
\max_{u,v\in\mathcal A}
\min\bigl\{|c|:c\in\mathcal A^*,\ ucv\in\mathcal L(X)\bigr\}
$$
is finite.  Also put
$
c_0=\frac{1}{2|\mathcal A|^2}.
$
Choose $\delta>0$ so that
$
\delta<\min\left\{\frac{\eta}{4},\frac{\eps}{16}\right\}.
$
If $h_\mu(\sigma)>0$, we additionally require $\delta<h_\mu(\sigma)/2$.

For a function $\psi\in C(X)$, write
$$
S_n\psi(x)=\sum_{j=0}^{n-1}\psi(\sigma^j x).
$$
By the Birkhoff ergodic theorem and the Shannon--McMillan--Breiman theorem \cite{Walters1982}, for all sufficiently large $n$ the set
$$
\begin{aligned}
T_n
=
\Bigl\{x\in X:\;&
\left|\frac1n S_n\psi_i(x)-a_i\right|<\delta
~\text{for every }1\le i\le s,
\\
&\text{and}\quad
\mu([x_0x_1\cdots x_{n-1}])
\le
\exp\bigl(-n(h_\mu(\sigma)-\delta)\bigr)
\Bigr\}
\end{aligned}
$$
has measure greater than $1/2$.  We now choose one such $n$, large enough in addition that
\begin{equation}\label{mp:eq:local-approx-large-n}
 n>2r,
 \qquad
 \frac{4rC}{n}<\delta,
 \qquad
 \frac{2CD}{n}<\frac{\eps}{8},
 \qquad
 \frac{D\htop(\sigma,X)-\log c_0}{n}<\frac{\eta}{4}.
\end{equation}
When $h_\mu(\sigma)>0$, we also require
\begin{equation}\label{mp:eq:local-approx-positive-count}
 n(h_\mu(\sigma)-\delta)+\log c_0>0.
\end{equation}

Let $\mathcal W_n$ be the set of admissible words $w$ of length $n$ for which $[w]\cap T_n\ne\varnothing$.  For every $w\in\mathcal W_n$, choose a point $x^w\in[w]\cap T_n$.  If $y\in[w]$, then
$$
\psi_i(\sigma^j y)=\psi_i(\sigma^j x^w)
\qquad
(r\le j\le n-r-1).
$$
Thus at most $2r$ terms in the two length-$n$ Birkhoff sums can differ, and each difference has absolute value at most $2C$.  By \eqref{mp:eq:local-approx-large-n},
\begin{equation}\label{mp:eq:uniform-typical-word}
\left|
\frac1nS_n\psi_i(y)-a_i
\right|
<
\delta+\frac{4rC}{n}
<2\delta
\end{equation}
for every $w\in\mathcal W_n$, every $y\in[w]$, and every $1\le i\le s$.  

For $u,v\in\mathcal A$, define
$$
\mathcal W_n(u,v)
=
\{w=w_0\cdots w_{n-1}\in\mathcal W_n:w_0=u,\ w_{n-1}=v\}.
$$
The sets $\mathcal W_n(u,v)$ form a partition of $\mathcal W_n$.  Since
distinct words of length $n$ determine disjoint cylinders and
$T_n\subset\bigcup_{w\in\mathcal W_n}[w]$, we have
$$
\sum_{(u,v)\in\mathcal A^2}
\mu\left(\bigcup_{w\in\mathcal W_n(u,v)}[w]\right)
=
\mu\left(\bigcup_{w\in\mathcal W_n}[w]\right)
\geq \mu(T_n)
>
\frac12.
$$
Therefore there exists $(a,b)\in\mathcal A^2$ such that, with
$\mathcal V_n:=\mathcal W_n(a,b)$,
$
\mu\left(\bigcup_{w\in\mathcal V_n}[w]\right)
\geq
\frac{1}{2|\mathcal A|^2}
=c_0.
$
Write $q=|\mathcal V_n|$.  The Shannon--McMillan--Breiman condition in the definition of $T_n$ gives
$$
\mu([w])
\le
\exp\bigl(-n(h_\mu(\sigma)-\delta)\bigr)
\qquad\text{for every }w\in\mathcal V_n.
$$
Consequently,
\begin{equation}\label{mp:eq:typical-word-count}
q
\ge
c_0\exp\bigl(n(h_\mu(\sigma)-\delta)\bigr).
\end{equation}

Choose a fixed word $c$ of length $\ell\le D$ such that $bca$ is
admissible.  For $w\in\mathcal V_n$, put
$
B_w=wc,
~
k=n+\ell.
$
Since every word in $\mathcal V_n$ begins with $a$ and ends with $b$,
arbitrary bi-infinite concatenations of the blocks $B_w$ are admissible in
$X$.

Define a zero-one transition matrix $\widehat A$ indexed by the finite alphabet
$
\widehat{\mathcal A}
=
\mathcal V_n\times\{0,1,\ldots,k-1\}
$
as follows: 
$$
\widehat A_{(w,j),(w',j')}
=
\begin{cases}
1,& w=w',\ 0\le j<k-1,\ \text{and }j'=j+1,\\
1,& j=k-1\ \text{and }j'=0,\\
0,& \text{otherwise}.
\end{cases}
$$
Let $\widehat Y=X_{\widehat A}$.
For
$
\widehat y=((w_t,j_t))_{t\in\mathbb Z}\in\widehat Y,
$
we call $j_t\in\{0,\ldots,k-1\}$ the \emph{phase} of $\widehat y$ at time
$t$.  Thus phase $0$ marks the beginning of an aligned block $B_{w_t}$.
Given any two
states $(w,j),(w',j')\in\widehat{\mathcal A}$, the prescribed transitions
give an admissible chain
$$
(w,j),\ldots,(w,k-1),(w',0),\ldots,(w',j').
$$
Consequently,
$
\bigl(\widehat A^{\,k-j+j'}\bigr)_{(w,j),(w',j')}>0,
$
so $\widehat A$ is irreducible.  Write
$B_w=B_w(0)\cdots B_w(k-1)$ and define the one-block code
$$
\pi:\widehat Y\longrightarrow X
$$
by
$$
\bigl(\pi(\widehat y)\bigr)_t=B_{w_t}(j_t)
\qquad
\text{for }
\widehat y=((w_t,j_t))_{t\in\mathbb Z}\in\widehat Y.
$$
The transition rules ensure that $\pi(\widehat y)$ is a shift of a
bi-infinite concatenation of the blocks $B_w$, and hence belongs to $X$.
Thus
$
Y:=\pi(\widehat Y)
$
is a sofic subshift of $X$.  Moreover, $\widehat Y$ is transitive because
$\widehat A$ is irreducible, and transitivity passes to factor images.
Therefore $Y$ is transitive.

For every $m\ge1$ and every $(w_1,\ldots,w_m)\in\mathcal V_n^m$, the word
$
B_{w_1}B_{w_2}\cdots B_{w_m}
$
of length $mk$ belongs to the language of $Y$.  Distinct $m$-tuples give distinct words, because the first $n$ symbols in each aligned length-$k$ block recover the corresponding word $w_j$.  Since $|\mathcal V_n|=q$, it follows that
$$
|\mathcal L_{mk}(Y)|\geq q^m
\qquad\text{for every }m\geq1.
$$
For every subshift, topological entropy is given by the word-complexity
formula (see \cite[Chapter~4]{LindMarcus2021})
$
\htop(\sigma,Y)
=
\lim_{N\to\infty}\frac1N\log|\mathcal L_N(Y)|.
$
Taking the subsequence $N=mk$ and using the preceding lower bound yields
\begin{equation}\label{mp:eq:sofic-entropy-lower}
\htop(\sigma,Y)
=
\lim_{m\to\infty}\frac1{mk}\log|\mathcal L_{mk}(Y)|
\geq
\lim_{m\to\infty}\frac1{mk}\log q^m
=
\frac1k\log q.
\end{equation}
If $h_\mu(\sigma)>0$, then \eqref{mp:eq:local-approx-positive-count}, \eqref{mp:eq:typical-word-count}, and $k\le n+D$ imply
$
\htop(\sigma,Y)
\ge
\frac{n(h_\mu(\sigma)-\delta)+\log c_0}{n+D}.
$
Hence
$$
\begin{aligned}
h_\mu(\sigma)-\htop(\sigma,Y)
\le
\frac{D h_\mu(\sigma)+n\delta-\log c_0}{n+D}
\le
\delta+\frac{D\htop(\sigma,X)-\log c_0}{n}
<\frac{\eta}{2}.
\end{aligned}
$$
If $h_\mu(\sigma)=0$, then simply
$
\htop(\sigma,Y)\ge0>h_\mu(\sigma)-\frac{\eta}{2}.
$
Thus in all cases
\begin{equation}\label{mp:eq:sofic-near-measure-entropy}
\htop(\sigma,Y)
>
h_\mu(\sigma)-\frac{\eta}{2}.
\end{equation}

We next prove 
\begin{equation}\label{mp:eq:all-sofic-measures-local}
\mathcal M_\sigma(Y)\subset U.
\end{equation}
Suppose first that $y\in Y$ has a lift in $\widehat Y$ whose phase at time zero is $0$.  Then the next $k$ symbols of $y$ form one of the blocks $B_w=wc$.  By \eqref{mp:eq:uniform-typical-word},
$$
\left|S_n\psi_i(y)-na_i\right|<2\delta n.
$$
On the connector part there are $\ell$ terms, each differing from $a_i$ by at most $2C$.  Therefore
\begin{equation}\label{mp:eq:aligned-block-average}
\left|
\frac1kS_k\psi_i(y)-a_i
\right|
\le
2\delta+\frac{2C\ell}{k}
\le
2\delta+\frac{2CD}{n}
=:\theta.
\end{equation}

Fix $y\in Y$ and a lift $\widehat y\in\pi^{-1}(y)$.  For each
$N\ge1$, the phase coordinate of $\widehat y$ partitions the index
interval $[0,N)$ into a collection of complete aligned intervals of
length $k$, together with at most two boundary intervals whose total
length is less than $2k$.  On every complete aligned interval,
\eqref{mp:eq:aligned-block-average} bounds the corresponding deviation
from $a_i$ by $\theta k$.  The contribution of the boundary intervals
is bounded by $2C$ per index, since $|\psi_i|,|a_i|\le C$.  Summing over
this partition and dividing by $N$ therefore yields
\begin{equation}\label{mp:eq:uniform-long-average}
\left|
\frac1N S_N\psi_i(y)-a_i
\right|
\le
\theta+\frac{4Ck}{N}.
\end{equation}
Let $\nu\in \mathcal M_\sigma(Y)$.  Invariance yields
$
\int\psi_i\,d\nu
=
\frac1N\int S_N\psi_i\,d\nu.
$
Integrating \eqref{mp:eq:uniform-long-average} and then letting $N\to\infty$, we obtain
$
\left|
\int\psi_i\,d\nu-a_i
\right|
\le\theta.
$
By the choices of $\delta$ and $n$ in \eqref{mp:eq:local-approx-large-n},
$
\theta
=
2\delta+\frac{2CD}{n}
<\frac{\eps}{4}.
$
Consequently,
$$
\begin{aligned}
\left|
\int\varphi_i\,d\nu-
\int\varphi_i\,d\mu
\right|
\le
2\|\varphi_i-\psi_i\|_\infty
+
\left|
\int\psi_i\,d\nu-
\int\psi_i\,d\mu
\right|<
\frac{\eps}{4}+\frac{\eps}{4}
=
\frac{\eps}{2}
<\eps.
\end{aligned}
$$
Thus $\nu\in\mathcal U_0\subset U$, proving \eqref{mp:eq:all-sofic-measures-local}.

Finally, Marcus's inside entropy-approximation theorem for sofic shifts \cite[Proposition~3]{Marcus1985} gives an SFT
$
Z\subset Y
$
with
$
\htop(\sigma,Z)
>
\htop(\sigma,Y)-\frac{\eta}{2}.
$
By \cite[Definition~4.4.2 and
Theorem~4.4.4]{LindMarcus2021} there is a transitive SFT
$\Lambda\subset Z$ such that
$
\htop(\sigma,\Lambda)=\htop(\sigma,Z).
$
Together with \eqref{mp:eq:sofic-near-measure-entropy}, this gives
$
\htop(\sigma,\Lambda)
>
h_\mu(\sigma)-\eta.
$
Since $\Lambda\subset Z\subset Y$, we also have
$
\mathcal M_\sigma(\Lambda)
\subset
\mathcal M_\sigma(Y)
\subset U.
$
\end{proof}

\subsection*{Acknowledgments}
X. Hou is supported by the National Natural Science Foundation of China (No. 12401231) and the Fundamental Research Funds for the Central Universities (No. DUT25RC(3)106). W. Lin is supported by the National Natural Science Foundation of China (No. 124B2010). X. Tian is supported by the National Natural Science Foundation of China (No. 12471182).

\noindent\textbf{Conflict of interest.} The authors declare that there is no conflict of interest.

\noindent\textbf{Data availability.} No data were used for the research described in the article.

\phantomsection
\addcontentsline{toc}{section}{References}
\bibliographystyle{plain}

\end{document}